\documentclass[a4paper, 11pt]{article}

\usepackage{amsmath,amssymb,amsthm}
\usepackage[latin1]{inputenc}
\usepackage{version,tabularx,multicol}
\usepackage{graphicx,float,psfrag}
\usepackage{stmaryrd}
 \usepackage{enumerate}
 \usepackage{color}
\usepackage[pdftex]{hyperref}
\usepackage{authblk}
\usepackage{tikz}
\usetikzlibrary{trees, shapes, arrows, positioning}
\usetikzlibrary{positioning, decorations.pathreplacing}
\usepackage[numbers,sort&compress]{natbib}

\theoremstyle{plain}
\newtheorem{theorem}{Theorem}[section]

 \newtheorem{corollary}[theorem]{Corollary}

 \newtheorem{proposition}[theorem]{Proposition}

\newtheorem{remark}{Remark}[section]

\newcommand{\vt}{{\Vert}}

\newcommand{\bm}[1]{\mbox{\boldmath $#1$}}

  \makeatletter
  \@addtoreset{equation}{section}
  \makeatother

 \def\beqlb{\begin{eqnarray}}\def\eeqlb{\end{eqnarray}}
 \def\beqnn{\begin{eqnarray*}}\def\eeqnn{\end{eqnarray*}}

 \def\mbb{\mathbb}

 \def\qed{\hfill$\Box$\medskip}

\newcommand{\bcen}{\begin{center}}
\newcommand{\ecen}{\end{center}}
\newcommand{\bgeqn}{\begin{equation}}
\newcommand{\edeqn}{\end{equation}}

\begin{document}

\title{Limit theorems for a supercritical multi-type branching process with immigration in a random environment}

\author{  Jiangrui Tan\thanks{%
 School of Mathematics and Statistics, Central South University, Changsha 410075, P.R. China. Email:tanjiangrui@csu.edu.cn
} }
\maketitle

\noindent{\bf Abstract}\quad
Let $\{Z_n^i = (Z_n^i(r))_{1 \le r \le d}: n \ge 0\}$ be a supercritical $d$-type branching process in an i.i.d. environment $\xi = (\xi_0, \xi_1, \dots)$, starting from a single particle of type $i$. The offspring distribution at generation $n$ depends on the environment $\xi_n$, and we denote by $M_n = (M_n(i,j))_{1 \le i,j \le d}$ the corresponding (random) mean matrix. Recently, Grama et al. (Ann. Appl. Probab. \textbf{33}(2023) 1213-1251) extended the famous Kesten--Stigum theorem to the random environment case with $d>1$. They improved upon previous work by innovatively constructing a new normalized population process $(\tilde{W}^i_n)$. Under several simple assumptions, they proved that $\tilde{W}^i_n$ converges almost surely to a limit $\tilde{W}^i$, and that $\tilde{W}^i$ is non-degenerate if and only if a $\mbb{E}X\log^+ X<\infty$ type condition holds. 

In this paper, we study the situation where an immigrant vector $Y_n$ joins the population $Z_n^i$ at each generation $n \ge 0$; the distribution of $Y_n$ also depends on the environment $\xi_n$. Following the approach of Grama et al., we construct a normalized process $(W^i_n)$ for the model with immigration, establishing a Kesten--Stigum type theorem that characterizes the non-degeneracy of its almost sure limit. Moreover, we provide complete $L^p$-convergence criteria for $(W^i_n)$, treating separately the cases $1 < p < \infty$ and $0 < p < 1$. As an important byproduct, a sufficient condition for the boundedness of the maximal function $\sup_n \tilde{W}_n^i$ is also obtained. Our results show that, under a mild restriction on the number of immigrants, the inclusion of immigration does not affect the almost sure convergence property of the original normalized process, but it does have an impact on the criterion for $L^p$ convergence. 

 \vspace{0.3cm}



\bigskip

\section{Introduction}
\vspace{3mm}
  Branching processes constitute a cornerstone of modern probability theory, providing a rich and flexible framework for modeling the evolution of populations where individuals reproduce and die independently. The analytical power of these models has led to their widespread application across diverse fields, including biology for modeling population dynamics and gene propagation, physics for particle cascade phenomena, and epidemiology for the spread of diseases, see for example \cite{KA,pake}.

The fundamental mechanics of a branching process are defined by its reproduction law. In its simplest discrete-time form which is refered as Bienaym\'{e}--Galton--Watson process, the process evolves as follows: starting from an initial set of particles, each particle in a given generation, independently of others, produces a random number of offspring according to a fixed probability distribution upon its death; the subsequent generation is then formed by the totality of these offspring. This simple yet powerful mechanism allows for a deep analysis of population behavior, such as extinction probabilities and growth rates.

A significant generalization was introduced by Smith and Wilkinson \cite{smith}, who pioneered the study of Branching Processes in a Random Environment (BPRE). In this model, the reproduction law itself is driven by an external, randomly evolving environment, making the offspring distributions generation-dependent and stochastic. This framework more accurately captures the fluctuating conditions real populations face. The theory was substantially advanced by Athreya and Karlin \cite{ath1,ath2}, who established a comprehensive foundation for BPREs under the assumption of a stationary and ergodic environment, revealing profound connections between the properties of the environment sequence and the long-term fate of the population. For continuous state BPRE and spatial BPRE, we refer to \cite{rr1,rr2,rr3,rr4,rr5,rr6} and the references therein.

Branching processes in random environments are further categorized into single-type and multi-type models. The single-type case, where all individuals are considered identical, has been extensively studied, and its asymptotic behavior is now relatively well-understood across subcritical, critical, and supercritical regimes. We refer the reader to the recent book by Kersting and Vatutin \cite{kersting} for more details.  In contrast, the multi-type scenario, which accounts for different classes of individuals with type-specific reproduction, presents considerably greater analytical challenges. For the multi-type BPRE in the critical and subcritical regimes, Key \cite{key} and Roitershtein \cite{ro} provided pivotal results characterizing the path to extinction and establishing the limit theorems. For the supercritical multi-type BPRE, a breakthrough was achieved by Grama et al. \cite{grama1}, who established the asymptotic composition and the limiting distribution of the process after a suitable normalization.

This paper is devoted to the study of a supercritical multi-type branching process with immigration in a random environment.  Our primary objective is to generalize the seminal work of Grama et al. \cite{grama1,grama2} by incorporating an immigration component, where each generation may receive an influx of new individuals from an external source.  A detailed description of the model's evolution is provided in Section \ref{sec2.1}. The inclusion of immigration is not merely a mathematical generalization but is crucial for modeling populations. For instance, Bansaye \cite{ba} investigated a single-type BPRE with immigration to deal with cell contamination problems, and Vatutin \cite{va} considered a multi-type BPRE with immigration to study polling systems with random regimes of service. For other results of single-type BRRE, continuous state BPRE and super-Brownian motion with immigration, we refer to \cite{re1,re2,re3,re4,re5,re6} and the references therein.

To clearly articulate our motivation and situate our contributions within the existing literature, we will now provide a detailed overview of the key elements and results from the study by Grama et al. \cite{grama1}.

Consider a $d$-type branching process in a deterministic environment initiated by a single particle of type $i$ : $\{Z^i_n=(Z^i_n(r))_{1\le r\le d}:n\ge0\}$, where $Z^i_n(r)$ denotes the number of particles of types $r$ at generation $n$ and $Z_0=e_i$, where $e_i$ is the unit vector whose $i$-th component is $1$. Let $M=(M(i,j))_{1\le i,j\le d}$ denote the mean progeny matrix of $\{Z_n^i\}_{n\ge0}$, that is, $M(i,j)=\mbb{E}(Z^i_1(j)|Z_0=e_i)$. We assume $M$ is primitive in the sense that there exists an integer $k\ge1$ such that $M^k(i,j)>0$ for all $i,j$. Let $\rho$ be the spectral radius of the
mean matrix of $M$. Let $u=(u(r))_{1\le r \le d}$ and $v=(v(r))_{1\le r \le d}$ be the associated right and left eigenvectors respectively with the normalization $\vt v\vt=\langle u,v\rangle=1$, where $\vt \cdot\vt$ denotes the $L_1$-norm and $\langle\cdot,\cdot\rangle$ the scalar product. Assume $\rho>1$, which means $\{Z_n^i\}_{n\ge0}$ is supercritical. The famous Kesten--Stigum theorem states that, for any $1\le i,j \le d$, there exists a nonnegative random variable $W^i$ such that
\[
\frac{Z^i_n(j)}{\rho^nv(j)}\xrightarrow{a.s.} W^iu(j),\quad \text{as}\quad n\to\infty.
\]
Moreover, $W^i$ is non-degenerate if and only if $\sup_i\mbb{E}(Z^i_1(j)\log^+Z_1^i(j))<\infty$, where $\log^{+}x:=0\vee\log x$. The proof of Kesten--Stigum theorem is based on the construction of the fundamental non-negative martingale
\begin{equation}\label{mar1}
W^i_0=1, \quad W_n^i=\frac{\langle Z_n^i, u \rangle}{\rho^n u(i)},\quad n\ge1.
\end{equation}

The generalization of Kesten--Stigum theorem to a random environment case is a long-standing problem. In this situation, the mean matrix and related spectral radius is random. For $n\ge0$, let $M_n=M_n(\xi)$ denote the matrix of conditioned means of the offspring distribution of $n$-th generation given the environment $\xi=(\xi_k)_{k\ge0}$:
\[
M_n(i,j)=\mbb{E}(Z_{n+1}^i(j)|Z_n=e_i,\xi).
\]
Denote by $\rho_n=\rho_n(\xi)$ the spectral radius of $M_n$. Denote the products by $M_{n,n+k}=M_n\cdots M_{n+k}$ and $\rho_{n,n+k}=\rho_n\cdots\rho_{n+k}$. Let $U_{n,n+k}$ and $V_{n,n+k}$ be the corresponding normalized right and left eigenvectors of $M_{n,n+k}$ respectively. The straightforward way to generalize Kesten--Stigum theorem would be replacing $\rho^n$ and $u$ by $\rho_{0,n-1}$ and $U_{0,n-1}$ in \eqref{mar1}. However, this construction does not lead to a martingale which makes the problem much more delicate. Recently, a breakthrough was achieved by Grama et al. \cite{grama1}. Grama and his collaborators utilized Henion's work \cite{hen} on the analog of the Perron-Frobenius theorem for products of random matrices to identify a suitable martingale construction, thereby providing a complete proof of the Kesten--Stigum theorem for multi-type BPRE. Under a mild condition, Henion's results \cite[Theorem 1]{hen} established the existence of a sequence of unit vector $(U_{n,\infty})_{n\ge0}$ satisfying
\[
M_nU_{n+1,\infty}=\lambda_nU_{n,\infty}, \quad n\ge0,
\]
where $\lambda_n=\vt M_nU_{n+1,\infty}\vt$. Let $\lambda_{n,n+k}=\lambda_n\cdots\lambda_{n+k}$. Therefore, Grama et al. could define a positive martingale associated with $(Z_n^i)$ as
\[
\tilde{W}_0^i=1, \quad \tilde{W}_n^i=\frac{\langle Z_n^i, U_{n,\infty} \rangle}{\lambda_{0,n-1} U_{0,\infty}(i)}, \quad n\ge1,
\]
and prove the non-degeneracy of its almost sure limit. Note that this construction is the same as \eqref{mar1} when the environment is deterministic.

The results of \cite{grama1} open ways to prove many important properties such as central limit theorems and large deviation asymptotics for multi-type BPRE, see \cite{grama2, grama3, grama4}. The main objective of this paper is to extend the Kesten--Stigum theorem to the framework of multi-type BPRE with immigration based on the work of Grama et al. \cite{grama1} and Henion \cite{hen}. We now explain briefly the main results of this paper. Let $\{X^i_n=(X^i_n(r))_{1\le r\le d}:n\ge0\}$ be a multi-type BPRE with immigration in an i.i.d. environment $\xi=(\xi_n)_{n\ge0}$. The random variable $X_{n+1}^i$, representing the particle count at generation $n$, comprises two components: the offspring generated by the $n$-th generation and the immigrant particles arriving at the $n$-th generation:
\[
X^i_{n+1}=\sum_{r=1}^d\sum_{l=1}^{X^i_n(r)}N_{l,n}^r+Y_n.\]
where $N_{l,n}^r\in\mbb{N}^d$ represents the offspring vector at time $n+1$ of the $l$-th particle of type $r$ in the generation $n$ and $Y_n$ represents the vector of particles immigrated at generation $n$. By grouping all immigrant particles and their descendants, we obtain the decomposition:
\begin{equation}\label{de2}
X^i_{n+1}=Z^i_{n+1}+\sum_{k=0}^n\hat{Z}_{k},
\end{equation}
 where $\hat{Z}_{k},k\ge0$ counts all descendants up to generation $n+1$ originating from the immigrants arriving at generation $k$, while $Z^i_{n+1}$ represents the descendants at generation $n+1$ from the initial particle of type $i$. Clearly, the process $\{Z^i_n\}_{n\ge0}$ is a multi-type BPRE. We will adopt the notation associated with $\{Z^i_n\}_{n\ge0}$ introduced earlier in this chapter. Define
\[
W_0^i=1, \quad W_n^i=\frac{\langle X_n^i, U_{n,\infty} \rangle}{\lambda_{0,n-1} U_{0,\infty}(i)}, \quad n\ge 1.
\]
The main contributions of this paper are as follows:
\begin{itemize}
    \item We first establish that, under mild conditions, $(W_n^i)_{n\ge0}$ forms a submartingale under both the quenched and annealed laws.
    \item Secondly, we characterize the necessary and sufficient conditions for the uniform boundedness of $(W_n^i)_{n\ge0}$ under the quenched and annealed laws, respectively. This directly implies the almost sure convergence $W^i := \lim_{n \to \infty} W^i_n$ ($\mathbb{P}$-a.s.) with $W^i$ taking values in $[0,\infty)$.
    \item Thirdly, under the aforementioned boundedness condition, we prove a Kesten--Stigum type theorem that characterizes the non-degeneracy of the limit $W^i$.
    \item Furthermore, we provide complete $L^p$-convergence criteria for $(W_n^i)$, treating the cases $1 < p < \infty$ and $0 < p < 1$ separately.
    \item As key byproducts that are instrumental to our proofs, we derive a decomposition of $W^i$ via a multi-type Galton--Watson tree and establish a sufficient condition for the boundedness of the maximal function $\sup_{n\ge0} \tilde{W}_n^i$.
\end{itemize}

\section{Preliminaries}
\vspace{3mm}
\subsection{Description of the model and notation}\label{sec2.1}
At first, let us fix some notation. Throughout this paper, we consider a $d$-type ($d>1$) branching process in a random environment (BPRE) with immigration. Let $\mathbb{R}^d$ denote the $d$-dimensional Euclidean space consisting of vectors with real coordinates. We write $\bm{1}$, $\bm{0}$, and $e_i$ for the $d$-dimensional vectors whose entries are all $1$, all $0$, and all $0$ except that the $i$-th component equals $1$, respectively. For any $x, y \in \mathbb{R}^d$, define
\[
\vt x \vt=\sum_{i=1}^d|x(i)|\quad \text{and }\quad \langle x, y\rangle=\sum_{i=1}^dx(i)y(i).
\]
Let $\mathcal{M}_d(\mathbb{R})$ denote the space of all $d \times d$ real matrices. For $M = (M(i,j))_{1 \le i, j \le d} \in \mathcal{M}_d(\mathbb{R})$, its operator norm is denoted by $\vt M \vt$ and defined as
\[
\vt M \vt=\sup_{\vt x\vt=1}\vt Mx\vt.
\]
We denote the set of non-negative integers by $\mathbb{N} = \{0, 1, 2, \dots\}$ and the set of positive integers by $\mathbb{N}_+ = \{1, 2, \dots\}$. The letters $C$, $C'$ (possibly with subscripts) denote strictly positive constants whose values may change from line to line, unless otherwise stated. The indicator function of an event $A$ is written as $\mathbf{1}_A$. Almost sure convergence is denoted by $\xrightarrow{\text{a.s.}}$.

Let \(\xi = (\xi_0, \xi_1, \xi_2, \dots)\) be an independent and identically distributed (i.i.d.) sequence of random variables taking values in an abstract space \(\Theta\), which serves as the random environment. Each realization of \(\xi_n\) corresponds to \(d+1\) probability distributions on \(\mathbb{N}^d\): \(d\) offspring distributions and one immigration distribution, all of which are characterized by their probability generating functions (p.g.f.s). Specifically, for \(s = (s_1, \dots, s_d) \in [0,1]^d\), the offspring distributions are defined by
\[
f_n^r(s) = \sum_{k_1, \dots, k_d = 0}^\infty p_{k_1, \dots, k_d}^r(\xi_n) \, s_1^{k_1} \cdots s_d^{k_d}, \quad r = 1, 2, \dots, d,
\]
and the immigration distribution is defined by
\[
g_n(s) = \sum_{k_1, \dots, k_d = 0}^\infty \hat{p}_{k_1, \dots, k_d}(\xi_n) \, s_1^{k_1} \cdots s_d^{k_d}.
\]

A \(d\)-type branching process in the random environment \(\xi\) with immigration is a process \(\bigl\{ X_n = (X_n(r))_{1 \leq r \leq d} : n \geq 0 \bigr\}\) defined as follows: the initial value \(X_0 \in \mathbb{N}^d\) is fixed, and for all \(n \ge 0\),
\begin{equation}\label{e1.1}
X_{n+1} = \sum_{r=1}^d \sum_{\ell=1}^{X_n(r)} N_{\ell, n}^r + Y_n.
\end{equation}
Here, \(X_n(r)\) denotes the number of particles of type \(r\) in generation \(n\), \(N_{\ell, n}^r \in \mathbb{N}^d\) is the offspring vector produced by the \(\ell\)-th type-\(r\) particle in generation \(n\), and \(Y_n \in \mathbb{N}^d\) is the immigration vector introduced at generation \(n\). Conditional on the environment \(\xi\), the random variables \(\{N_{\ell, n}^r : \ell \ge 1, n \ge 0, 1 \le r \le d\}\) and \(\{Y_n : n \ge 0\}\) are all mutually independent. For each \(\ell \ge 1\) and \(n \ge 0\), the vector \(N_{\ell, n}^r\) has the p.g.f. \(f_n^r\), and \(Y_n\) has the p.g.f. \(g_n\). Note that we do not assume independence between the offspring distribution \(p(\xi_n)\) and the immigration distribution \(\hat{p}(\xi_n)\). When the process starts with a single particle of type \(i\), i.e., \(X_0 = e_i\), we denote \(X_n\) by \(X_n^i\) for all \(n \ge 0\).

The process \(\{X_n = (X_n(r))_{1 \leq r \leq d} : n \geq 0\}\) with an environment $\xi$ evolves as follows.
\begin{itemize}
    \item[(i)] (Offspring production). Every individual of type \(i\;(1\le i\le d)\) in generation \(n\) has a lifetime of one unit. At the end of its life, it produces a random offspring vector \(j = (j_1, \dots, j_d) \in \mathbb{N}^d\) with probability \(p_{j_1, \dots, j_d}^i(\xi_n)\).
    \item[(ii)] (Immigration). Simultaneously, at the end of generation \(n\), an immigrant vector \(Y_n = (Y_n(1), \dots, Y_n(d))\) arrives. The distribution of \(Y_n\) is governed by the law \(\hat{p}(\xi_n)\). All reproduction and immigration events are independent of one another. The individuals of generation \(n+1\) consist precisely of the offspring produced in (i) and the immigrants arriving in (ii).
\end{itemize}

Let $\mathbb{P}_{\xi}$ denote the conditional probability law under which the process $\{X_n: n \ge 0\}$ is defined given the environment $\xi$. The total probability $\mathbb{P}$ is expressed as
\[
\mathbb{P}(\mathrm{d}x, \mathrm{d}\xi) = \mathbb{P}_{\xi}(\mathrm{d}x) \, \tau(\mathrm{d}\xi),
\]
where $\tau$ denotes the distribution of the environment sequence $\xi$. The measure $\mathbb{P}_{\xi}$ is commonly referred to as the quenched law, while $\mathbb{P}$ is called the annealed law. Let $Y = (Y_0, Y_1, \dots)$ denote the immigration sequence, and let $\mathbb{P}_{\xi,Y}$ be the conditional probability under $\mathbb{P}$ given $\xi$ and $Y$: for any measurable set $A$,
\[
\mathbb{P}_{\xi,Y}(A) = \mathbb{E}[\mathbf{1}_A \mid \sigma(\xi, Y)],
\]
where $\sigma(\xi, Y)$ is the $\sigma$-algebra generated by $\xi$ and $Y$. We denote the expectations with respect to $\mathbb{P}_{\xi,Y}$, $\mathbb{P}_{\xi}$, and $\mathbb{P}$ by $\mathbb{E}_{\xi,Y}$, $\mathbb{E}_{\xi}$, and $\mathbb{E}$, respectively.

Given the random environment $\xi$, we define the sequence of mean matrices $(M_n)_{n \ge 0}$ by
\begin{equation}\label{mn}
M_n = M_n(\xi_n) := \left( \frac{\partial f_n^r}{\partial s_j}(\mathbf{1}) \right)_{1 \le r,j \le d} 
= \left( \mathbb{E}_{\xi} N_{1,n}^r(j) \right)_{1 \le r,j \le d},
\end{equation}
where $\frac{\partial f}{\partial s_j}(1)$ denotes the left derivative of $f$ with respect to $s_j$ at the point $\mathbf{1}$. The $(i,j)$-th entry $M_n(i,j)$ of $M_n$ represents the conditional expectation of the number of type-$j$ offspring produced by a single type-$i$ particle at generation $n$. Each matrix $M_n$ depends solely on $\xi_n$, and the sequence $(M_n)_{n \geq 0}$ is i.i.d. For $0 \leq k \leq n$, we define the product matrix
\[
M_{k,n} := M_k \cdots M_n = \left( \frac{\partial f^i_k \circ f_{k+1} \circ \cdots \circ f_n}{\partial s_j}(\mathbf{1}) \right)_{1 \leq i,j \leq d}.
\]
\subsection{Assumptions }
We now introduce the following condition for classifying multi-type BPREs with immigration:

\noindent\textbf{Condition H1.} The random matrix $M_0$ satisfies the moment condition:
\[
\mathbb{E}[\log^+\| M_0 \|] < \infty.
\]

From \cite{kesten}, under Condition \textbf{H1}, there exists a constant $\gamma<\infty$ such that
\[
\lim_{n \to \infty} \frac{1}{n} \log \| M_{0,n-1} \|=\gamma\]
with probability $1$, and
\[
\lim_{n \to \infty} \frac{1}{n} \mathbb{E}[\log \| M_{0,n-1} \|]=\gamma.
\]
The constant $\gamma$ is referred to as the Lyapunov exponent of the random matrix sequence $(M_n)_{n \ge 0}$. We say that a multi-type BPRE with immigration is subcritical if $\gamma<0$, critical if $\gamma=0$ and supercritical if $\gamma>0$. This classification aligns with the standard classification for both the non-immigration case and the case of deterministic environments. Throughout this paper, we assume $\gamma>0$.

In what follows, we review relevant results from Hennion~\cite{hen} and Grama et al.~\cite{grama1}. Let $\mathcal{S}$ denote the semigroup of matrices in $\mathcal{M}_d(\mathbb{R})$ with positive entries that are allowable in the sense that every row and column contains at least one strictly positive element, and let $\mathcal{S}_0$ be the subset of matrices with all entries strictly positive. We impose the following condition:

\noindent\textbf{Condition H2.} The matrix $M_0$ belongs to $\mathcal{S}$ $\mathbb{P}$-a.s., and
\[
\mathbb{P}\left(\bigcup_{n \ge 0} \{ M_{0,n} \in \mathcal{S}_0 \} \right) > 0.
\]

Under Condition \textbf{H2}, for any $n, k \geq 0$, the product $M_{n,n+k}$ belongs to $\mathcal{S}$ $\mathbb{P}$-a.s. Let $\rho_{n,n+k}$ denote the spectral radius of $M_{n,n+k}$. By the classical Perron--Frobenius theorem (see, e.g., \cite{Athreya1972Branching}), $\rho_{n,n+k}$ is a strictly positive eigenvalue of $M_{n,n+k}$, associated with positive right and left eigenvectors $U_{n,n+k}$ and $V_{n,n+k}$, respectively, normalized such that $\|U_{n,n+k}\| = 1$ and $\langle V_{n,n+k}, U_{n,n+k} \rangle = 1$. 

Hennion~\cite{hen} extended the Perron--Frobenius theory to products of random matrices. Under Condition \textbf{H2}, he showed that for all $n \ge 0$, the sequence $(U_{n,n+k})_{k \ge 0}$ converges $\mathbb{P}$-a.s. to a random unit vector, denoted here by $U_{n,\infty}$. Furthermore, defining
\[
\lambda_n = \lambda_n(\xi) = \| M_n U_{n+1,\infty} \|, \quad n \ge 0,
\]
which are strictly positive, the following relation holds:
\begin{equation}\label{rel}
M_n U_{n+1,\infty} = \lambda_n U_{n,\infty}.
\end{equation}
The numbers $\lambda_n$ are called the {pseudo spectral radii} of the random matrix products and play a key role in constructing martingales associated with multi-type BPREs. Moreover, from \cite[Proposition 2.5]{grama1}, under conditions \textbf{H1} and \textbf{H2}, 
\begin{equation}\label{1.21}
\lim_{n\to\infty}\frac{1}{n}\log \lambda_{0,n-1}=\mbb{E}\log\lambda_0=\gamma\quad \mbb{P}\text{-a.s}.
\end{equation}
Recall from~\eqref{de2} that a multi-type BPRE with immigration admits a decomposition into a branching process without immigration and a sequence of descendants generated by immigrant particles (see Section~3 for a detailed derivation):
\begin{equation}\label{mul}
X^i_{n+1} = Z^i_{n+1} + \sum_{k=0}^n \hat{Z}_k,
\end{equation}
where $\{Z_n^i: n \ge 0\}$ is a multi-type BPRE with $Z_0^i = e_i$. Define
\[
\tilde{W}_0^i = 1, \quad \tilde{W}_n^i = \frac{\langle Z_n^i, U_{n,\infty} \rangle}{\lambda_{0,n-1} U_{0,\infty}(i)},
\]
and the filtration
\[
\tilde{\mathcal{F}}_n = \sigma\left( \xi, \; N_{l,k}^r: 0 \leq k \leq n-1, \; 1 \leq r \leq d, \; l \geq 1 \right).
\]

We also require the following condition introduced by Furstenberg and Kesten~\cite{kesten}:

\noindent\textbf{Condition H3.} There exists a constant $D> 1$ such that
\[
1 \le \frac{\max_{1 \le i, j \le d} M_0(i,j)}{\min_{1 \le i, j \le d} M_0(i,j)} \le D, \quad \mathbb{P}\text{-a.s.}
\]

Note that Condition \textbf{H3} implies Condition \textbf{H2}. The next condition is analogous to the classical $\mathbb{E}[X \log^+ X] < \infty$ assumption in the Kesten--Stigum theorem~\cite{ks} for the deterministic environment case:

\noindent\textbf{Condition H4.} For all $1 \le i, j \le d$,
\[
\mathbb{E}\left[ \frac{Z_1^i(j)}{M_0(i,j)} \log^+ \frac{Z_1^i(j)}{M_0(i,j)} \right] < \infty.
\]

Under the aforementioned assumptions, the results of Grama et al. \cite[Theorems 2.3 and 2.9]{grama1} and \cite[Theorem 2.1]{grama2}  can be specified to the following two propositions.

\begin{proposition}[{\cite[Theorems 2.3 \& 2.9]{grama1}}]\label{pro1}
  (a) Assume that Condition \textbf{H2} holds. For each $1 \le i \le d$, the sequence $(\tilde{W}_n^i)_{n \ge 0}$ is a non-negative martingale with respect to the filtration $(\tilde{\mathcal{F}}_n)_{n \ge 0}$, under both $\mathbb{P}_\xi$ and $\mathbb{P}$. Consequently, $(\tilde{W}_n^i)_{n \ge 0}$ converges $\mathbb{P}$-almost surely to a non-negative random variable $\tilde{W}^i$.
  
  (b) Assume Conditions \textbf{H1}, \textbf{H3}, and $\gamma > 0$. Then if and only if Condition \textbf{H4} is satisfied, $\tilde{W}^i$ ($1 \le i \le d$) are non-degenerate, in the sense that
  \[
  \mathbb{P}_\xi(\tilde{W}^i > 0) > 0 \quad \mathbb{P}\text{-a.s.}
  \]
  and $\mathbb{E}_\xi[\tilde{W}^i] = 1$ for all $1 \le i \le d$.
\end{proposition}
The proposition above establishes the almost sure convergence of $\tilde{W}^i$ and provides a criterion for non-degeneracy. Next, we will present the $L^p$ convergence results for $\tilde{W}^i$. Before doing so, we first introduce some notation. Set
\[
I = \{s \leq 0 : \mathbb{E}M_0(i,j)^s < +\infty \quad \forall i,j=1,\dots,d\}.
\]
By H\"{o}lder's inequality, $I$ is an interval, and if there exists $s \in I$ with $s < 0$, then $M_0 > 0$ $\mathbb{P}$-a.s. From \cite[Proposition 3.1]{grama2},  for $s \in I$, the limit
\begin{equation}\label{2.1}
\kappa(s) := \lim_{n \to +\infty} (\mathbb{E}\|M_{0,n-1}\|^s)^{1/n}
\end{equation}
exists, with $\kappa(s) < +\infty$. The following relations (see \cite[Lemma 3.2 \& (4.10)]{grama2}) will be used repeatedly: for all $s\in I$, $n\ge1$ and $1\le i \le d$, there exist $C_s>0$ such that
\begin{align}\label{2.2}
  &\mbb{E}\vt M_{0,n-1}\vt^s\le\mbb{E}_{T^n\xi}\lambda_{0,n-1}^s=\mbb{E}_{T^n\xi}\vt M_{0,n-1}U_{n,\infty} \vt^s \le C_s \kappa(s)^n\quad \text{and}\nonumber\\
  &\mbb{E}\vt M_{0,n-1}\vt^s\le\mbb{E}_{T^n\xi}(\lambda_{0,n-1}U_{0,\infty})^s=\mbb{E}_{T^n\xi}\langle M_{0,n-1}U_{n,\infty}, e_i\rangle^s\le C_s \kappa(s)^n,
\end{align}
where $T^n$ is the $n$-fold iteration of $T$ (with the convention $T^0$ is the identity operator), with $T$ the shift translation on the environment:
\[
T(\xi_0,\xi_1,\dots)=(\xi_1,\xi_2,\dots). 
\]
\begin{proposition}[{\cite[Theorem 2.1]{grama2}}]\label{pro2} Assume Conditions \textbf{H1}, \textbf{H3}, and $\gamma > 0$. Assume $p>1$, $1-p\in I$. If and only if 
\[
\max_{1 \leq i, j \leq d} \mathbb{E} \left( \frac{Z_1^i(j)}{M_0(i,j)} \right)^p < +\infty \quad \text{and} \quad \kappa(1-p) < 1,
\] 
then $W_n^i\to W^i$ in $L^p$ as $n\to\infty$.
\end{proposition}

\section{Main results}

The following conclusions hold for all types of the initial particle; henceforth we fix an arbitrary type \(i\) with \(1 \le i \le d\).

From now on, we always assume Conditions \textbf{H1}, \textbf{H3}, and \(\gamma > 0\). Consider the natural filtration defined by \( \mathcal{F}_0 = \sigma(\xi) \) and, for \( n \ge 1 \),
\[
\mathcal{F}_n = \sigma\left(\xi, Y_k, N_{l,k}^r, \, 0 \leq k \leq n-1, \, 1 \leq r \leq d, \, l \ge 1\right).
\]
Recall that for \(1\le i \le d\), the process \((W_n^i)_{n \ge 0}\) is defined as
\begin{equation}\label{def}
W_0^i = 1, \quad W_n^i = \frac{\langle X_n^i, U_{n,\infty} \rangle}{\lambda_{0,n-1}U_{0,\infty}(i)}, \quad n \ge 1.
\end{equation}

The following theorem presents the main results of this paper. It provides a boundedness condition and establishes both almost sure and \(L^1\) convergence of the process \((W_n^i)_{n \ge 0}\) under the measure \(\mathbb{P}_{\xi,Y}\).
\begin{theorem}\label{th1}
$(\mathrm{a})$ For \(\mathbb{P}\)-almost every \(\xi\) and \(Y\), the sequence \((W^i_n, \mathcal{F}_n)\) is a submartingale under \(\mathbb{P}_{\xi,Y}\), and satisfies
\begin{equation}\label{th1e1}
\mathbb{E}_{\xi,Y}W_n^i = 1 + \sum_{k=0}^{n-1} \frac{\langle Y_k,U_{k+1,\infty}\rangle}{{\lambda_{0,k}U_{0,\infty}(i)}}.
\end{equation}
$(\mathrm{b})$ The condition \(\mathbb{E}\log^+ \frac{|Y_0|}{\lambda_0} < \infty\) is necessary and sufficient for \(\sup_{n\ge0}\mathbb{E}_{\xi,Y}(W_n^i) < \infty\) \(\mathbb{P}\)-almost surely. When it holds, the limit \(W^i := \lim_{n \to \infty} W^i_n\) exists \(\mathbb{P}\)-almost surely and takes values in \([0,\infty)\).\\
$(\mathrm{c})$ Assume \(\mathbb{E}\log^+ \frac{|Y_0|}{\lambda_0} < \infty\). Then the following statements are equivalent:
\begin{itemize}
    \item Condition \textbf{H4} is satisfied;
    \item \(W^i\) is non-degenerate, i.e., \(\mathbb{P}(W^i>0) > 0\) and \(\mathbb{E}_{\xi}W^i \ge 1\) \(\mathbb{P}\)-almost surely;
    \item \(\mathbb{E}_{\xi,Y}|W^i - W_n^i| \to 0\) \(\mathbb{P}\)-almost surely.
\end{itemize}
\end{theorem}

Our result extends the Kesten--Stigum type theorem for multi-type branching processes in a random environment \cite[Theorem 2.9]{grama1} to the setting with immigration. It shows that, under the mild immigration condition \(\mathbb{E}\log^+ (|Y_0|/\lambda_0) < \infty\), the classical condition of the  form \(\mathbb{E}[X \log^+ X] < \infty\) on the offspring distribution   remains both necessary and sufficient for the martingale limit \(W^i\) to be non-degenerate. Furthermore, the non-degeneracy of \(W^i\) is equivalent to the \(L^1\) convergence of the sequence \((W_n^i)\).

On the other hand, Theorem \ref{th1} also generalizes the results of Wang and Liu \cite[Theorem 3.2, 4.3 and 5.2]{wang} and Bansaye \cite[Proposition 1(iii)]{ba}, who studied the single-type branching process in a random environment with immigration. By adopting the martingale construction introduced in \cite{grama1}, we find that the limiting behavior of \(W^i\) in the multi-type case parallels that in the single-type setting.

Using similar arguments and proof techniques with Theorem \ref{th1}, the next two corollaries provide boundedness conditions \(L^1\) convergence of the process \((W_n^i)_{n \ge 0}\) under the measures \(\mathbb{P}_{\xi}\) and \(\mathbb{P}\), respectively.
\begin{corollary}\label{cor1}
$(\mathrm{a})$ If \(\mathbb{E}_{\xi}|Y_0| < \infty\) for \(\mathbb{P}\)-almost every \(\xi\), then \((W_n, \mathcal{F}_n)\) is a submartingale under \(\mathbb{P}_{\xi}\) for almost all \(\xi\), and
\begin{equation}\label{th1e2}
\mathbb{E}_{\xi}W_n^i = 1 + \sum_{k=0}^{n-1} \frac{\langle \mathbb{E}_{\xi}Y_k,U_{k+1,\infty}\rangle}{{\lambda_{0,k}U_{0,\infty}(i)}}.
\end{equation}
Furthermore, \(\sup_{n\ge0}\mathbb{E}_{\xi}(W_n^i) < \infty\) holds \(\mathbb{P}\)-almost surely if and only if \(\mathbb{E}\log^+ \frac{\mathbb{E}_{\xi}|Y_0|}{\lambda_0} < \infty\).\\
$(\mathrm{b})$ Assume \(\mathbb{E}\log^+ \frac{\mathbb{E}_{\xi}|Y_0|}{\lambda_0} < \infty\). Then \(\mathbb{E}_{\xi}|W^i - W_n^i| \to 0\) \(\mathbb{P}\)-almost surely if and only if Condition \textbf{H4} is satisfied.
\end{corollary}
\begin{remark}
   Observe that the function $\log^+ x$ can be slightly modified (for instance, as in \cite[(5.4)]{wang}) to become concave on $[0,\infty)$. Using such a concave modification, one can show that the condition
$\mathbb{E}\log^+ \frac{\mathbb{E}_{\xi}|Y_0|}{\lambda_0} < \infty$ implies
$\mathbb{E}\log^+ \frac{|Y_0|}{\lambda_0} < \infty.$ Consequently, the uniform boundedness of $(W_n^i)$ under $\mathbb{P}_{\xi}$ implies its uniform boundedness under $\mathbb{P}_{\xi,Y}$.
\end{remark}
\begin{corollary}\label{cor2}
$(\mathrm{a})$ If $-1\in I$, \(\kappa(-1) < 1\) and \(\mathbb{E}\frac{|Y_0|}{\lambda_0} < \infty\), then \((W_n, \mathcal{F}_n)\) is a submartingale under \(\mathbb{P}\) and satisfies
\begin{equation}\label{th1e3}
\mathbb{E}W_n^i = 1 + \sum_{k=0}^{n-1} \mathbb{E}\left( \frac{\langle Y_k,U_{k+1,\infty}\rangle}{{\lambda_{0,k}U_{0,\infty}(i)}} \right) < \infty.
\end{equation}
Conversely, if \(\sup_{n\ge0}\mathbb{E}(W_n^i) < \infty\), then necessarily \(\kappa(-1) \le 1\) and \(\mathbb{E}\frac{|Y_0|}{\lambda_0} < \infty\).\\
$(\mathrm{b})$ Assume $-1\in I$, \(\kappa(-1) < 1\) and \(\mathbb{E}\frac{|Y_0|}{\lambda_0} < \infty\). Then \(\mathbb{E}|W^i - W_n^i| \to 0\) if and only if Condition \textbf{H4} holds.
\end{corollary}
\begin{remark}
 Differing from the single-type setting, we obtain only the weaker condition $\kappa(-1) \le 1$ (instead of $\kappa(-1) < 1$) as necessary for the uniform boundedness of $(W^i_n)$ under $\mbb{P}$. This gap originates from the fact that, from equation \eqref{2.2}, we have the upper bound $\mathbb{E} \lambda_{0,n-1}^{s} \le C_s (\kappa(s))^n$ but lack a comparable lower bound. This defect will also appear in the subsequent $L^p$ convergence conclusions.
\end{remark}

 The $L^p~(1<p<\infty)$  convergence of BPRE without immigration (Proposition \ref{pro2}) plays a key role in studying the asymptotic behavior (e.g., large deviations and Berry--Esseen bounds) of \((\tilde{W}_n^i)\), as can be seen in the preprints \cite{grama3, grama4}. Motivated by this, we generalize the $L^p$ convergence result to the process \((W_n^i)\) involving immigration.

\begin{theorem}\label{th4}
Assume $\mathbb{P}(|Y_0|=0) < 1$, $1<p<\infty$ and $-p\in I$. If the conditions
\begin{equation}\label{2.5}
    \kappa(-p) < 1, \quad \max_{1\le i,j \le d} \mathbb{E}\left( \frac{Z_1^{i}(j)}{M_0(i,j)} \right)^{p} < \infty, \quad \text{and} \quad \mathbb{E}\frac{|Y_0|^p}{\lambda_0^p} < \infty
\end{equation}
hold, then
\[
\sup_{n \ge 0} \| W_n^i \|_p < \infty \quad \text{and} \quad \lim_{n \to \infty} \| W_n^i - W^i \|_p = 0.
\]
Conversely, if \(\sup_{n \ge 0} \| W_n^i \|_p < \infty\), then \begin{equation}\label{12.6}
    \kappa(-p) \le 1, \quad \kappa(1-p)<1 \quad \text{and} \quad \mathbb{E}\frac{|Y_0|^p}{\lambda_0^p} < \infty.
\end{equation}

\end{theorem}

\begin{remark}
    Compared with Proposition \ref{pro2}, the theorem above shows that, in addition to imposing a restriction on the number of immigrants, we require a stronger condition on the offspring distribution, specifically $\kappa(-p) < 1$ (note that $\kappa(-p) < 1$ implies $\kappa(1-p) < 1$ by Section \ref{sec5.2}), to establish the $L^p$ convergence of $W^i_n$. This indicates that, unlike almost sure convergence, the coming of immigration significantly affects the $L^p$ convergence of the process. Furthermore, this result generalizes the conclusion of Wang \cite[Theorem 6.2]{wang} in the single-type case, and our finding runs parallel to theirs.
\end{remark}
Finally, we establish the $L^p~(0<p<1)$  boundedness for the process \((W_n^i)\) involving immigration. Notice that, in the case without immigration, the process \((\tilde{W}_n^i)\) is automatically bounded by the fact that $\mbb{E}\tilde{W}_n^i=1$.
\begin{theorem}\label{th5}
Assume $\mathbb{P}(|Y_0|=0) < 1$, $0<p < 1$ and $-p\in I$. If the conditions
\begin{equation}\label{2.6}
    \kappa(-p) < 1 \quad \text{and} \quad \mathbb{E}\frac{|Y_0|^p}{\lambda_0^p} < \infty
\end{equation}
hold, then $\sup_{n \ge 0} \| W_n^i \|_p < \infty$. Conversely, if \(\sup_{n \ge 0} \| W_n^i \|_p < \infty\), then \eqref{2.6} holds with the first inequality relaxed to \(\kappa(-p) \le 1\).
\end{theorem}

\section{Auxiliary results}
\subsection{The Maximal function of the martingale $(\tilde{W}^i_n)$}
Recall that for the multi-type BPRE $\{Z_n^i=(Z^i_n(j))_{1\le j \le d}: n \ge 0\}$ with $Z_0^i = e_i$, the process
\[
\tilde{W}_0^i = 1, \qquad \tilde{W}_n^i = \frac{\langle Z_n^i, U_{n,\infty} \rangle}{\lambda_{0,n-1} U_{0,\infty}(i)}, \quad n \ge 1,
\]
is, by Proposition \ref{pro1}, a non-negative martingale under both $\mathbb{P}_{\xi}$ and $\mathbb{P}$ with respect to the filtration $(\tilde{\mathcal{F}}_n)_{n \ge 0}$, where
\[
\tilde{\mathcal{F}}_n = \sigma\left( \xi, \; N_{l,k}^r: 0 \leq k \leq n-1, \; 1 \leq r \leq d, \; l \geq 1 \right).
\]

Define $\tilde{W}^{i\ast}=\sup_{n\ge0}\tilde{W}_n^{i}$ for $1\le i \le d$. As an extension for \cite[Theorem 1.3]{Liu2}, we give a sufficient condition for the boundedness of $\tilde{W}^{i*}$, which is crucial in extending the Kesten--Stigum theorem to multi-type branching processes in random environments with immigration.

\begin{proposition}\label{pro3}
    Assume Conditions \textbf{H1}, \textbf{H3} and $\gamma > 0$. If $\kappa(-1)<1$ and Condition \textbf{H4} is satisfied, then
    \[
    \mathbb{E}[ \tilde{W}^{i*} ] < \infty \quad \text{for all } 1 \le i \le d.
    \]
\end{proposition}
\begin{proof}

For $n \ge 1$, define the martingale difference
\begin{align}
D_n = D_n^i &:=\tilde{W}_{n+1}^i - \tilde{W}_n^i \nonumber \\
&= \sum_{r=1}^d \frac{U_{n+1,\infty}(r)}{\lambda_{0,n} U_{0,\infty}(i)} \sum_{j=1}^d \sum_{l=1}^{Z_n^i(j)} N_{l,n}^j(r) - \tilde{W}_n^i \nonumber\\
&= \sum_{j=1}^d \frac{U_{n,\infty}(j)}{\lambda_{0,n-1} U_{0,\infty}(i)} \sum_{l=1}^{Z_n^i(j)} \left( \frac{\langle N_{l,n}^j, U_{n+1,\infty} \rangle}{\lambda_n U_{n,\infty}(j)} - 1 \right) \nonumber\\
&= \sum_{j=1}^d \frac{U_{n,\infty}(j)}{\lambda_{0,n-1} U_{0,\infty}(i)} \sum_{l=1}^{Z_n^i(j)} (L_{l,n}^j - 1), \label{em1b}
\end{align}
where we set
\[
L_{l,n}^j := \frac{\langle N_{l,n}^j, U_{n+1,\infty} \rangle}{\lambda_n U_{n,\infty}(j)}.
\]
Note that $\mathbb{E}_{\xi} L_{l,n}^j = 1$ by \eqref{mn}. The supremum can then be written as
\[
\tilde{W}^{i\ast} = 1 + \sup_{n \ge 1}\, (D_1 + D_2 + \cdots + D_n).
\]

Define the auxiliary function
\[
\ell(x) = 
\begin{cases} 
1 - \dfrac{1}{2x}, & x > 1, \\[0.8em]
\dfrac{x}{2},      & x \le 1.
\end{cases}
\]
Then $\phi(x) := x \ell(x)$ is convex, and $\phi_{1/2}(x) := \phi(x^{1/2})$ is concave (hence subadditive). Applying \cite[Part (b), Lemma 2.1]{Liu2} yields
\begin{equation}\label{em3}
\mathbb{E}\, \phi\bigl(\tilde{W}^{i\ast}-1\bigr) \le C \sum_{n \ge 1} \mathbb{E}\, \phi(|D_n|).
\end{equation}

Let $\mathbb{E}_n[\cdot] = \mathbb{E}[\,\cdot \mid \tilde{\mathcal{F}}_n]$. From \eqref{em1b} and the elementary inequality $(\sum_{k=1}^d x_k)^2 \le C \sum_{k=1}^d x_k^2$, we have
\begin{align*}
\mathbb{E}_n \phi(|D_n|) &= \mathbb{E}_n \phi_{1/2}(D_n^2) \\
&\le \mathbb{E}_n \phi_{1/2}\!\left[ C \sum_{j=1}^d \Bigl( \sum_{l=1}^{Z_n^i(j)} \hat{L}_{l,n}^j \Bigr)^2 \right],
\end{align*}
where
\[
\hat{L}_{l,n}^j := \frac{U_{n,\infty}(j)}{\lambda_{0,n-1} U_{0,\infty}(i)} \, (L_{l,n}^j - 1), \qquad l \ge 1.
\]
Using the subadditivity of $\phi_{1/2}$ gives
\begin{equation}\label{em2}
\mathbb{E}_n \phi(|D_n|) \le C\, \mathbb{E}_n \sum_{j=1}^d \phi\!\left( \Bigl| \sum_{l=1}^{Z_n^i(j)} \hat{L}_{l,n}^j \Bigr| \right).
\end{equation}

Under $\mathbb{P}_\xi$, the variables $\hat{L}_{l,n}^j$ ($l \ge 1$) are independent with zero mean. Hence, conditioned on $\tilde{\mathcal{F}}_n$, the family $\{\hat{L}_{l,n}^j : 1 \le l \le Z_n^i(j)\}$ forms a sequence of martingale differences with respect to a suitable natural filtration. Applying \cite[Lemma 2.1, part b]{Liu2} to \eqref{em2} yields
\begin{align*}
\mathbb{E}_n \phi(|D_n|) &\le C\, \mathbb{E}_n \sum_{j=1}^d \sum_{l=1}^{Z_n^i(j)} \phi\bigl( |\hat{L}_{l,n}^j| \bigr).
\end{align*}
Recalling $\phi(x)=x\ell(x)$ and that $Z_n^i$ is $\tilde{\mathcal{F}}_n$-measurable, we obtain
\begin{equation}\label{eh1}
\mathbb{E}_n \phi(|D_n|) \le C\, \mathbb{E}_n \sum_{j=1}^d \frac{Z_n^i(j)U_{n,\infty}(j)}{\lambda_{0,n-1} U_{0,\infty}(i)} \, \bigl|L_{l,n}^j - 1\bigr| \, \ell\bigl(|\hat{L}_{l,n}^j|\bigr).
\end{equation}

Because $\mathbb{E}_{\xi} \sum_{j=1}^d \frac{Z_n^i(j)U_{n,\infty}(j)}{\lambda_{0,n-1} U_{0,\infty}(i)} = 1$, we have the almost-sure bound
\[
\frac{Z_n^i(j)U_{n,\infty}(j)}{\lambda_{0,n-1} U_{0,\infty}(i)} \le 1 \quad \mathbb{P}_\xi\text{-a.s.}
\]
Taking expectations in \eqref{eh1} and combining with \eqref{em3} leads to
\begin{equation}\label{key-ineq}
\mathbb{E}\, \phi\bigl(\tilde{W}^{i\ast}-1\bigr) \le C \sum_{n=1}^{\infty} \mathbb{E} \sum_{j=1}^d \bigl|L_{l,n}^j - 1\bigr| \, \ell\bigl(|\hat{L}_{l,n}^j|\bigr).
\end{equation}
Fix $b \in (\kappa(-1), 1)$. We split the expectation in \eqref{key-ineq} according to the two regions $\{(\lambda_{0,n-1} U_{0,\infty}(i))^{-1} \le b^n\}$ and $\{(\lambda_{0,n-1} U_{0,\infty}(i))^{-1} \ge b^n\}$:
\begin{align*}
\mathbb{E}\, \phi\bigl(\tilde{W}^{i\ast}-1\bigr) &= \mathbb{E}\, \phi\bigl(\tilde{W}^{i\ast}-1\bigr) \mathbf{1}_{\{(\lambda_{0,n-1} U_{0,\infty}(i))^{-1} \le b^n\}} \;+\; \mathbb{E}\, \phi\bigl(\tilde{W}^{i\ast}-1\bigr) \mathbf{1}_{\{(\lambda_{0,n-1} U_{0,\infty}(i))^{-1} \ge b^n\}} \\
&=: I_1(n) + I_2(n).
\end{align*}

\noindent \textbf{Estimate of $I_1(n)$.} Since $\ell$ is increasing on $[0,\infty)$ and $|\hat{L}_{l,n}^j| \le |L_{l,n}^j-1| (\lambda_{0,n-1} U_{0,\infty}(i))^{-1}$ , we have
\[
I_1(n) \le C \sum_{n=1}^{\infty} \mathbb{E} \sum_{j=1}^d \bigl|L_{l,n}^j - 1\bigr| \, \ell\bigl( |L_{l,n}^j - 1| \, b^n \bigr).
\]
To bound $|L_{l,n}^j-1|$, observe first that
\[
L_{l,n}^j = \frac{\langle N_{l,n}^j, U_{n+1,\infty} \rangle}{\langle M_n[j], U_{n+1,\infty} \rangle},
\]
where $M_n[j]$ denotes the $j$-th row of $M_n$. Because $\|U_{n+1,\infty}\|=1$, we obtain, almost surely,
\begin{align*}
L_{l,n}^j &\le \frac{\max_{1 \le k \le d} N_{l,n}^j(k)}{\min_{1 \le k \le d} M_n(j,k)} \le \frac{D\, \|N_{l,n}^j\|}{\displaystyle\max_{1 \le i,j \le d} M_n(i,j)},
\end{align*}
where the last inequality follows by Condition \textbf{H3}. Consequently,
\begin{align*}
|L_{l,n}^j-1| &\le \Bigl| \frac{D\, \|N_{l,n}^j\|}{\max_{i,j} M_n(i,j)} - 1 \Bigr| =: \overline{X}_{n,j}.
\end{align*}
Since the environment sequence is i.i.d., $(\overline{X}_{n,j})_{n\ge1}$ are i.i.d. for each fixed $j$; we denote a generic copy by $\overline{X}_{1,j}$. Using that $\ell(t) < 1$ for $t \ge 1$ and $\ell$ is increasing, we deduce
\begin{align*}
I_1(n) &\le C \sum_{n=1}^{\infty} \mathbb{E} \sum_{j=1}^d \overline{X}_{1,j} \, \ell\bigl( \overline{X}_{1,j} b^n \bigr) \\
&\le C \sum_{j=1}^d \sum_{n=1}^{\infty} \mathbb{E}\, \overline{X}_{1,j} \int_{b^n \overline{X}_{1,j}}^{b^{n-1} \overline{X}_{1,j}} \frac{\ell(t)}{t} \,  \mathrm{d}t \\
&= C \sum_{j=1}^d \mathbb{E}\, \overline{X}_j \left( \int_0^1 \frac{\ell(t)}{t} \,  \mathrm{d}t + \int_1^{\overline{X}_{1,j}} \frac{\ell(t)}{t} \, \mathrm{d}t \right) \\
&\le C \sum_{j=1}^d \mathbb{E}\, \overline{X}_{1,j} \bigl( 1 + \log^+ \overline{X}_{1,j} \bigr).
\end{align*}
Condition \textbf{H4} guarantees $\mathbb{E}\, L_{l,n}^j \log^+ L_{l,n}^j < \infty$, which is equivalent to $\mathbb{E}\, \overline{X}_{1,j} \log^+ \overline{X}_{1,j} < \infty$. Hence $I_1(n) < \infty$.

\noindent \textbf{Estimate of $I_2(n)$.} Using again $\ell(t) < 1$ for $t \ge 1$,
\begin{align*}
I_2(n) &\le C \sum_{n=1}^{\infty} \mathbb{E} \sum_{j=1}^d \bigl|L_{l,n}^j - 1\bigr| \, \mathbf{1}_{\{(\lambda_{0,n-1} U_{0,\infty}(i))^{-1} \ge b^n\}} \\
&\le C \sum_{n=1}^{\infty} \sum_{j=1}^d \mathbb{E}\, \overline{X}_{1,j} \; \mathbb{P}\bigl((\lambda_{0,n-1} U_{0,\infty}(i))^{-1} \ge b^n \bigr) \\
&\le C \sum_{j=1}^d \mathbb{E}\, \overline{X}_{1,j} \sum_{n=1}^{\infty} b^{-n} \mathbb{E}\, (\lambda_{0,n-1} U_{0,\infty}(i))^{-1} \\
&\le C \sum_{j=1}^d \mathbb{E}\, \overline{X}_{1,j} \sum_{n=1}^{\infty} \bigl[ b^{-1} \kappa(-1) \bigr]^n,
\end{align*}
where the last inequality follows by \eqref{2.2}. Because $b^{-1}\kappa(-1) < 1$, we conclude that $I_2(n)<\infty$.

Both $I_1(n)$ and $I_2(n)$ are finite, therefore $\mathbb{E}\, \phi(\tilde{W}^{i\ast}-1) < \infty$. This implies $\mathbb{E}\, \tilde{W}^{i\ast} < \infty$, completing the proof.
\end{proof}

\subsection{Decomposition for the martingale}

In this section, we generalizes the approach of \cite{wang} to the multi-type case, aiming to derive a decomposition of \((W_n^i)_{n\ge0}\) through the construction of a random tree. 

A multi-type branching process can be identified naturally as a random colored-tree (where type $i$ is regarded as color $i$), which is a subset of $$\tilde{U}:=\{1,\cdots,d\}\times\cup_{n=0}^{\infty}\mbb{N}_+^n$$ with the convention $\mbb{N}_+^0=\{\emptyset\}$. The initial particle with type $i$ is denoted by $(i,\emptyset)$. 

As usual, for two sequences $u,v\in\cup_{n=0}^{\infty}\mbb{N}_+^n$, denote by $|u|$ the length of $u$ and denote by $uv=(u,v)$ the sequence obtained by the juxtaposition of $u$ and $v$. By convention, $|\emptyset|=0$ and $\emptyset u=u$. A particle of type $i$ of generation $n$ is denoted by $(i,u)$ with $u\in\mbb{N}_+^n$ a sequence of length $n$; its $k$-th child of type $j$ is denoted $(j, uk)$, which is linked with its ancestor $(i, u)$. By abuse of notation, for an element $t=(i,u)\in\tilde{U}$ we write $r(t)=i$ for its type and $|t|:=|u|$ for its length.  
For $t_{1}=(r(t_{1}),u_{1}),\,t_{2}=(r(t_{2}),u_{2})\in\tilde{U}$ define  
\[
t_{1}t_{2}:=(r(t_{2}),\,u_{1}u_{2}).
\] 
For example, for $t_1=(2,13)$ and $t_2=(3,212)$, we have $r(t_1)=2, r(t_2)=3$, $|t_1|=2$, $|t_2|=3$ and $t_1t_2=(3,13212)$.   

For the Galton--Watson tree including immigrants with the initial particle $(i,\emptyset)$, we add a sequence of eternal particle \[
\{(i,0_{k}):k\ge 0\},
\qquad 0_{n}:=0_{n-1}0\ \ (\text{with }|0_{k}|=k).
\]
The eternal particle $(i,0_{n})$ at time $n$ produces $Y_{n}+e_{i}$ offspring:
 $$(i,0_{n+1}), (1,0_n1),\cdots,(1,0_nY_n(1)),(2,0_n1),\cdots,(2,0_nY_n(2)),\cdots,(d,0_nY_n(d)).$$  Consequently, the set $T^{i}$ of all particles of a multi-type branching process with immigration in a random environment, started by a single type-$i$ particle, can be decomposed into two disjoint trees (see Figure 1):
\begin{itemize}
    \item $\tilde{T}^{i}$ -- begins with the initial particle $(i,\emptyset)$ and contains all its descendants;
    \item $\hat{T}^{i}$ -- begins with the eternal particle $(i,0_{0})$ and contains all its descendants (including all immigrant particles).
\end{itemize}
  
   \begin{figure}[h]
\centering
\begin{tikzpicture}[
    level distance=1.2cm,
    sibling distance=0.5cm,
    every node/.style={ inner sep=1pt},
    edge from parent/.style={draw, -},
    label distance=-5pt
]

\node (T) at (0,4.2) [draw=none] {$T^i$};

\node (00) at (-4, 2.8) {$(i,0_0)$};
\node[draw=none] (k1) at (-4, 3.5) {$\hat{T}^i$}; 

\coordinate (00_child) at (-4, 1.5);
\node (01) at (-7, 1.5) {$(i,0_1)$} edge (00);
\node (001) at (-5.5, 1.5) {$(1,0_{0}1)$} edge (00);
\node (00Y0) at (-3, 1.5) {$(1,0_0Y_0(1))$} edge (00);
\node (00Y01) at (0, 1.5) {$(d,0_0Y_0(d))$} edge (00);
\node at (-4.4,1.5) {$\cdots$}; 
\node at (-1.5,1.5) {$\cdots$};
\node (02) at (-7, 0.2) {$(i,0_2)$} edge (01);
\node (01Y1) at (-5.05, 0.2) {$(d,0_1Y_1(d))$} edge (01);
\node (01Y2) at (-2.45, 0.2) {$(1,0_0Y_0(d)1)$} edge (00Y01);
\node (01Y2) at (0.55, 0.2) {$(d,0_0Y_0(d)Y_1(d))$} edge (00Y01);
\node at (-3,1.2) (vertical_dots) {$\vdots$}; 
\node at (-5.5,1.2) (vertical_dots) {$\vdots$}; 
\node at (-6.2,0.2) {$\cdots$}; 
\node at (-3.8,0.2) {$\cdots$};
\node at (-1.05,0.2) {$\cdots$};
\node at (-7,-0.1) (vertical_dots) {$\vdots$}; 
\node at (-5.05,-0.1) (vertical_dots) {$\vdots$}; 
\node at (-2.45,-0.1) (vertical_dots) {$\vdots$}; 
\node at (0.55,-0.1) (vertical_dots) {$\vdots$}; 

\node (Phi) at (4, 2.8) {$(i,\emptyset)$} ;
\node[draw=none] (k2) at (4, 3.5) {$\tilde{T}^i$}; 
\draw [ decorate, decoration={brace, amplitude=5pt}]
      (k1.north east) -- (k2.north west)
      node [midway, below=6pt, font=\footnotesize] {};

\node (1) at (1.5, 1.5) {$(1,1)$} edge (Phi);
\node (2) at (3.45, 1.5) {$(1,X_{(i,\emptyset)}(1))$} edge (Phi);
\node (XPhi) at (6.5, 1.5) {$(d,X_{(i,\emptyset)}(d))$} edge (Phi);

\node at (2.2,1.5) {$\cdots$};
\node at (5,1.5) {$\cdots$};
\node at (1.5, 1.2) (vertical_dots) {$\vdots$}; 
\node at (3.45, 1.2) (vertical_dots) {$\vdots$}; 

\node (1) at (3.45, 0.2) {$\cdots$} edge (XPhi);
\node (2) at (6.5, 0.2) {$(d,X_{(i,\emptyset)}(1)X_{(d,X_{(i,\emptyset)}(d))}(d))$} edge (XPhi);
\node at (3.45, -0.1) (vertical_dots) {$\vdots$}; 
\node at (6.5, -0.1) (vertical_dots) {$\vdots$}; 

\node[draw=none] at (0, -1.5) {\textbf{Figure 1:} \quad Genealogical tree of a multi-type BPRE with immigration};

\end{tikzpicture}
\end{figure}

Let us now define more formally the family tree \(T^{i}\) (the set of all individuals) together with its two disjoint subtrees \(\tilde{T}^{i}\) and \(\hat{T}^{i}\) for each \(1\le i \le d\).

First, define the sets
\[
\hat{U}:=\{1,\dots ,d\}\times\{0_{k}u:k\ge0,\;u\in\bigcup_{n=1}^{\infty}\mathbb{N}^{n}\}
\quad\text{and}\quad
E_{i}:=\{i\}\times\{0_{k}:k\ge0\}.
\]
Clearly, \(E_{i}\subset\hat{U}\) is the set of eternal particles of type \(i\).

We introduce a family of random variables \(\{X_{a}:a\in\tilde{U}\cup\hat{U}\}\) such that, under \(\mathbb{P}_{\xi}\), they are mutually independent. Their distributions are constructed as follows:
\begin{itemize}
    \item For an eternal particle \(t=(i,0_{n})\in E_{i}\), let \(X_{t}=Y_{n}\) with distribution \(\hat{p}(\xi_{n})\) under \(\mathbb{P}_{\xi}\).
    \item For any other particle \(t\in\tilde{U}\cup\hat{U}\setminus E_{i}\), let \(X_{t}\) have distribution \(p^{\,r(t)}(\xi_{|t|})\).
\end{itemize}

Now we define the two trees recursively by generation.
For \(n\ge0\), let \(\tilde{T}^{i}_{n}\) be the set of individuals at time \(n\) descending from the initial particle \((i,\emptyset)\):
\[
\tilde{T}^{i}_{0}=\{(i,\emptyset)\},
\]
and for \(n\ge0\),
\[
\tilde{T}^{i}_{n+1}= \bigcup_{k=1}^{d}\bigl\{(k,uj): t=(\cdot,u)\in\tilde{T}^{i}_{n},\; 1\le j\le X_{t}(k)\bigr\}.
\]
For \(n\ge0\), let \(\hat{T}^{i}_{n}\) be the set of individuals at time \(n\) descending from the eternal particle \((i,0_{0})\):
\[
\hat{T}^{i}_{0}=\{(i,0_{0})\},
\]
and for \(n\ge0\),
\[
\hat{T}^{i}_{n+1}= \{(i,0_{n+1})\}\cup\bigcup_{k=1}^{d}\bigl\{(k,uj): t=(\cdot,u)\in\hat{T}^{i}_{n},\; 1\le j\le X_{t}(k)\bigr\}.
\]

Finally, set
\[
\tilde{T}^{i}= \bigcup_{n\ge 0}\tilde{T}^{i}_{n},
\qquad
\hat{T}^{i}= \bigcup_{n\ge 0}\hat{T}^{i}_{n}.
\]
Then \(\tilde{T}^{i}\) is the family tree of a multi-type branching process in the random environment \(\xi\), started by a single type-\(i\) particle and governed by the offspring laws \((p^{r}(\xi_{n}))_{1\le r\le d,\;n\ge0}\). The tree \(\hat{T}^{i}\) contains precisely all eternal particles, all immigrants, and all descendants of immigrants.

\noindent\textbf{Shifted trees.}
For a particle \(a\in\tilde{U}\cup\hat{U}\setminus E_{i}\) (a non-eternal particle), we define the shifted tree rooted at \(a\) as \(\tilde{T}[a]=\bigcup_{n\ge0}\tilde{T}_{n}[a]\), where
\[
\tilde{T}_{0}[a] = (r(a),\emptyset),\qquad
\tilde{T}_{n+1}[a] = \bigcup_{k=1}^{d}\bigl\{(k,uj): t=(\cdot,u)\in\tilde{T}_{n}[a],\; 1\le j\le X_{at}(k)\bigr\}.
\]

For an eternal particle \(a\in E_{i}\), the shifted tree of \(\hat{T}\) rooted at \(a\) is \(\hat{T}[a]=\bigcup_{n\ge0}\hat{T}_{n}[a]\), where
\[
\hat{T}_{0}[a] = (r(a),0_{0}),\qquad
\hat{T}_{n+1}[a] = \{(r(a),0_{n+1})\}\cup\bigcup_{k=1}^{d}\bigl\{(k,uj): t=(\cdot,u)\in\hat{T}_{n}[a],\; 1\le j\le X_{at}(k)\bigr\}.
\]

Note that the root of the shifted tree \(\tilde{T}[a]\) (respectively \(\hat{T}[a]\)) is \((r(a),\emptyset)\) (respectively \((r(a),0_{0})\)), which is distinct from the particle \(a\) itself. A shifted tree is therefore different from a subtree. Indeed, the subtree of \(\tilde{T}^{i}\) starting at a vertex \(a=(r(a),t)\in\tilde{T}^{i}\) is defined as
$a\tilde{T}^{i}:=\bigcup_{k=1}^{d}\{(k,tu):(k,u)\in\tilde{T}[a]\}.$

For a non-eternal particle \(a\in\tilde{U}\cup\hat{U}\setminus E_{i}\), define the \(n\)-th generation population vector of the shifted tree as
\[
\tilde{Z}_{n}[a]=(\tilde{Z}_{n}[a](k))_{1\le k\le d},\qquad
\tilde{Z}_{n}[a](k):=\#\{t\in\tilde{T}_{n}[a]: r(t)=k\}.
\]
For an eternal particle \(a\in E_{i}\), define
\[
\hat{Z}_{n}[a]=(\hat{Z}_{n}[a](k))_{1\le k\le d},\qquad
\hat{Z}_{n}[a](k):=\#\{t\in\hat{T}_{n}[a]: r(t)=k\}.
\]
Clearly, $\hat{Z}_{n}[a]$ conditioned on the environment $\xi$ has the same distribution as $\hat{Z}_{n}[r(a),\emptyset]$ conditioned on the environment $T^{|a|}\xi$. Therefore, the multi-type branching process with immigration in a random environment $\{Z^i_n;n\ge0\}$ admits the following decomposition:
\begin{equation}\label{de1}
    Z^i_n=\tilde{Z}_n[(i,\emptyset)]+\hat{Z}_n[(i,0_0)]-e_i.
\end{equation}

For $a\in\tilde{U}\cup\hat{U}\setminus E_i$, define $\tilde{W}_0[a]=1$ and
\[
\tilde{W}_n[a]=\frac{\langle \tilde{Z}_n[a], U_{n+|a|,\infty} \rangle}{\lambda_{|a|,|a|+n-1}U_{|a|,\infty}(r(a))}, \quad n \geq 1.
\]
For $a\in E_i$, define $\hat{W}_0[a]=1$ and 
\[
\hat{W}_n[a]=\frac{\langle \hat{Z}_n[a], U_{n+|a|,\infty} \rangle}{\lambda_{|a|,|a|+n-1}U_{|a|,\infty}(i)}, \quad n \geq 1.
\]
From \cite{grama1}, when $a=(i,\emptyset)$, $\{\tilde{W}_n[a];n\ge0\}$ is a nonnegative martingale, both under $\mbb{P}_{\xi}$ and $\mbb{P}$ with respect to the filtration $\tilde{\mathcal{F}}_n$. It is easily to see that, for $a\in\tilde{U}\cup\hat{U}\setminus E_i$, the law under $P_{\xi}$ of the process $\{\tilde{W}_n[a];n\ge0\}$ is the same as that under $\mbb{P}_{T^{|a|}\xi}$ of the process $\{\tilde{W}_n[(r(a),\emptyset)];n\ge0\}$.
Since the environment sequence is i.i.d., the process $\{\tilde{W}_n[a];n\ge1\}$ is identical distributed as $\{\tilde{W}_n[(r(a),\emptyset)];n\ge0\}$ under $\mbb{P}$. Therefore, from Proposition \ref{pro1}, the limit $\tilde{W}[a]:=\lim_{n\to\infty}\tilde{W}_n[a]$ exists almost surely.

The existence of $W^i:=\lim_{n\to\infty}W_n^i~\mbb{P}$-a.s. is guaranteed by part (b) of Theorem \ref{th1}. The proof of part (b) will be presented in the next section and does not rely on the material of this section. For the sake of reading continuity, the decomposition formula for $W^i$ is stated below.
\begin{proposition}\label{pro4}
    Assume Conditions \textbf{H1}, \textbf{H3} and $\gamma > 0$. If $\kappa(-1)<1$ and $\mbb{E}\log^+\frac{|Y_0|}{\lambda_0}<\infty$, then we have
   \begin{equation}\label{2.8}
     W^i=\tilde{W}[(i,\emptyset)]+\sum_{j=1}^{\infty}\sum_{k=1}^{d}\sum_{l=1}^{Y_{j-1}(k)}\tilde{W}[(k,0_{j-1}l)] \frac{U_{j,\infty}(k)}{\lambda_{0,j-1}U_{0,\infty}(k)}.
   \end{equation}

\end{proposition}
\begin{proof}
Since the offspring of the initial particle $(i,0_0)$ consist of $(i,0_1)$ and the immigrants, the definition of the random tree $(\hat{T}_n[a])_{n\ge0}$ for $a \in E_i$ yields
\[
\hat{Z}_n[(i,0_0)] = \hat{Z}_{n-1}[(i,0_1)] + \sum_{k=1}^{d}\sum_{l=1}^{Y_0(k)} \tilde{Z}_{n-1}[(k,0_0l)].
\]
Using the definitions of $\hat{W}_n[a]$ and $\tilde{W}_n[a]$, we obtain
\begin{align*}
    \hat{W}_n[(i,0_0)] = \frac{\hat{W}_{n-1}[(i,0_{1})] U_{1,\infty}(i)}{\lambda_0 U_{0,\infty}(i)} + \sum_{k=1}^{d}\sum_{l=1}^{Y_{0}(k)} \tilde{W}_{n-1}[(k,0_{0}l)] \frac{U_{1,\infty}(k)}{\lambda_{0} U_{0,\infty}(i)}.
\end{align*}
Iterating this relation gives
\[
\hat{W}_n[(i,0_0)] = \frac{U_{n,\infty}(i)}{\lambda_{0,n-1} U_{0,\infty}(i)} + \sum_{j=1}^n \sum_{k=1}^{d} \sum_{l=1}^{Y_{j-1}(k)} \tilde{W}_{n-j}[(k,0_{j-1}l)] \frac{U_{j,\infty}(k)}{\lambda_{0,j-1} U_{0,\infty}(i)}.
\]
Combining this with \eqref{de1}, we arrive at the decomposition
\begin{align}
  W_n^i &= \frac{\langle \tilde{Z}_n[(i,\emptyset)] + \hat{Z}_n[(i,0_0)] - e_i, U_{n,\infty} \rangle}{\lambda_{0,n-1} U_{0,\infty}(i)} \nonumber\\
  &= \tilde{W}_n[(i,\emptyset)] + \sum_{j=1}^n \sum_{k=1}^{d} \sum_{l=1}^{Y_{j-1}(k)} \tilde{W}_{n-j}[(k,0_{j-1}l)] \frac{U_{j,\infty}(k)}{\lambda_{0,j-1} U_{0,\infty}(i)}.\label{decompose}
\end{align}

Define $\tilde{W}^{\ast}[a] = \sup_{n \ge 0} \tilde{W}_n[a]$ for $a \in \tilde{U} \cup \hat{U} \setminus E_{i}$. If
\begin{equation}\label{em1}
\sum_{j=1}^{\infty} \sum_{k=1}^{d} \sum_{l=1}^{Y_{j-1}(k)} \tilde{W}^{\ast}[(k,0_{j-1}l)] \frac{U_{j,\infty}(k)}{\lambda_{0,j-1} U_{0,\infty}(i)} < \infty \quad \mathbb{P}\text{-a.s.},
\end{equation}
then the dominated convergence theorem applied to \eqref{decompose} yields
\[
W^i = \tilde{W}[(i,\emptyset)] + \sum_{j=1}^{\infty} \sum_{k=1}^{d} \sum_{l=1}^{Y_{j-1}(k)} \tilde{W}[(k,0_{j-1}l)] \frac{U_{j,\infty}(k)}{\lambda_{0,j-1} U_{0,\infty}(i)}.
\]

To verify \eqref{em1}, we consider the $\mathbb{P}_{\xi,Y}$-expectation of the summand. Condition \eqref{em1} is implied by the almost sure finiteness of
\[
\sum_{j=1}^{\infty} \lambda_{0,j-2}^{-1} \left( \sum_{k=1}^d \frac{Y_{j-1}(k) U_{j,\infty}(k)}{\lambda_{j-1} U_{0,\infty}(i)} \mathbb{E}_{\xi} \tilde{W}^{\ast}[(k,0_{j-1}1)] \right).
\]
By \cite[Lemma 3.1]{wang}, the series above converges almost surely if
\begin{align*}
& \mathbb{E} \log^+ \!\left( \sum_{k=1}^d \frac{Y_{j-1}(k) U_{j,\infty}(k)}{\lambda_{j-1} U_{0,\infty}(i)} \mathbb{E}_{\xi} \tilde{W}^{\ast}[(k,0_{j-1}1)] \right) \\
< \; & \mathbb{E} \log^+ \!\left[ \frac{|Y_{j-1}|}{\lambda_{j-1} U_{0,\infty}(i)} \left( \sum_{k=1}^d \mathbb{E}_{\xi} \tilde{W}^{\ast}[(k,0_{j-1}1)] \right) \right] < \infty.
\end{align*}
The hypothesis $\mathbb{E}\log^+(|Y_0|/\lambda_0) < \infty$ already handles the first factor. Thus, by $\log^+x\le x$ for $x>0$, a sufficient condition is
\[
\mathbb{E} \log^+ \!\left( \sum_{k=1}^d \mathbb{E}_{\xi} \tilde{W}^{\ast}[(k,0_{j-1}1)] \right) \le \sum_{k=1}^{d} \mathbb{E} \tilde{W}^{\ast}[(k,0_{j-1}1)] < \infty.
\]
Finally, Proposition \ref{pro3} guarantees that under the assumptions $\kappa(-1)<1$ and Condition \textbf{H5}, the bound above holds. This establishes \eqref{em1} and, together with \eqref{decompose}, completes the proof.
\end{proof}
\section{Proofs}
\subsection{Proof of Theorem \ref{th1}}
 
\textbf{Proof of $(\mathbf{a})$.} By the decomposition \eqref{e1.1} and induction, $Z_n$ and hence $W_n$ are $\mathcal{F}_n$-measurable for every $n\ge0$. Conditionally on the environment $\xi$, the families  
\[
\{N_{l,k}^r: l\ge1,\;0 \le k \le n-1,\; 1 \le r \le d\},\quad 
\{N_{l,n}^r: l\ge1,\; 1 \le r \le d\},
\]  
and vectors $(Y_n)$ are independent. Consequently, for any $A\in\mathcal{F}_n$,
\[
\mathbb{E}_{\xi,Y}\bigl[N_{l,n}^r(j)\mathbf{1}_A\bigr]
   =\mathbb{E}_{\xi,Y}\bigl[M_n(r,j)\mathbf{1}_A\bigr].
\]

Using \eqref{e1.1} and \eqref{def}, we therefore obtain  
\begin{align}
\mathbb{E}_{\xi,Y}\bigl(W^i_{n+1}\mid\mathcal{F}_n\bigr)
&= \mathbb{E}_{\xi,Y}\!\left(
      \frac{\sum_{r=1}^d\sum_{l=1}^{X_n(r)}\langle N_{l,n}^r,U_{n+1,\infty}\rangle
           +\langle Y_n,U_{n+1,\infty}\rangle}
           {\lambda_{0,n}U_{0,\infty}(i)}
      \;\bigg|\;\mathcal{F}_n\right)\nonumber\\[4pt]
&= \frac{\langle X_n,M_nU_{n+1,\infty}\rangle}
        {\lambda_{0,n}U_{0,\infty}(i)}
   +\frac{\langle Y_n,U_{n+1,\infty}\rangle}
        {\lambda_{0,n}U_{0,\infty}(i)}\nonumber\\[4pt]
&= W_n^i+\frac{\langle Y_n,U_{n+1,\infty}\rangle}
               {\lambda_{0,n}U_{0,\infty}(i)}.\label{3.11}
\end{align}
 Moreover,  
$\mathbb{E}_{\xi,Y}(W^i_{n+1}\mid\mathcal{F}_n)\ge W_n^i$ shows that $(W_n^i,\mathcal{F}_n)$ is a submartingale under $\mathbb{P}_{\xi,Y}$ for each $i$. Equation \eqref{th1e1} now follows by induction:
\begin{equation}\label{ew1}
\mathbb{E}_{\xi,Y}(W_n^i)
 =1+\sum_{k=0}^{n-1}
    \frac{\langle Y_k,U_{k+1,\infty}\rangle}
         {\lambda_{0,k}U_{0,\infty}(i)} .
\end{equation}

\noindent\textbf{Proof of $(\mathbf{b})$.} Because the environment sequence is i.i.d., $(Y_n)_{n\ge0}$ is i.i.d., while $(\lambda_n)_{n\ge0}$ and $(U_{n,\infty})_{n\ge0}$ are stationary and ergodic, respectively. Since $U_{k+1,\infty}>0$ and $\|U_{k+1,\infty}\|=1$ almost surely, $\mathbb{E}_{\xi,Y}(W_n^i)$ is bounded almost surely if and only if
\begin{equation}\label{3.17}
1+\sum_{k=0}^{n-1}\lambda_{0,k-1}^{-1}\frac{|Y_k|}{\lambda_{k}}<\infty \qquad \mbb{P}\text{-a.s.}        
\end{equation}

Recall from \cite[Proposition 2.5]{grama1} that under Conditions \textbf{H1} and \textbf{H2},  
$\mathbb{E}\log\lambda_0=\gamma\in(0,\infty)$. Hence, by \cite[Lemma 3.1]{wang}, if $\mathbb{E}\log^+ \frac{|Y_0|}{\lambda_0}<\infty$, then for every $1\le i \le d$,
\[
\sup_{n\ge1}\mathbb{E}_{\xi,Y}(W_n^i)
   =1+\sum_{k=0}^{\infty}
      \frac{\langle Y_k,U_{k+1,\infty}\rangle}
           {\lambda_{0,k}U_{0,\infty}(i)}
   <\infty \qquad \text{a.s.}
\]
Because $(W_n^i)_{n\ge0}$ is $L^1$-bounded under $\mathbb{P}_{\xi,Y}$, the martingale convergence theorem implies that the limit $W^i=\lim_{n\to\infty} W_n^i$ exists $\mathbb{P}_{\xi,Y}$-almost surely, taking values in $[0,\infty)$ for each $i$. As this holds for almost every realisation of $\xi$ and $Y$, we conclude that  
\[
W^i_n\to W^i\quad \mathbb{P}\text{-a.s.}
\]

Conversely, assume $\mathbb{E}\log^+ \frac{|Y_0|}{\lambda_0}=\infty$.  
Set $\phi_n:=\log^+ \frac{|Y_n|}{\lambda_n}$ and let $F(x):=\mathbb{P}(\phi_n\le x)$. For any fixed $A>0$,

\[
1+\sum_{n=1}^{\infty}\bigl[1-F(An)\bigr]
   \ge \int_{0}^{\infty}\bigl[1-F(x)\bigr]\,dx
   =\mathbb{E}\phi_1
   =\infty .
\]

By the Borel--Cantelli lemma, there is an infinite number of events $\phi_n\ge An$ with probability $1$. Hence $\lim_{n\to\infty}\phi_n/n\ge A$ almost surely. Observe that  
\[
\Bigl( \lambda_{0,n-1}^{-1}\frac{|Y_n|}{\lambda_{n}}\Bigr)^{1/n}
   =\exp\Bigl(\frac{\phi_n}{n}-\frac{1}{n}\log\lambda_{0,n-1}\Bigr).
\]
From \eqref{1.21} we have $\lim_{n\to\infty}\frac{1}{n}\log\lambda_{0,n-1}=\gamma\in(0,\infty)$ almost surely. Because $A$ can be chosen arbitrarily large,  
\[
\limsup_{n\to\infty}\Bigl( \lambda_{0,n-1}^{-1}\frac{|Y_n|}{\lambda_{n}}\Bigr)^{1/n}=\infty
\quad\text{a.s.},
\]
which implies $\sup_{n\ge1}\mathbb{E}_{\xi,Y}(W_n^i)=\infty$ almost surely.

\noindent\textbf{Proof of (c).} 
Assume first that Condition \textbf{H4} holds. By Proposition \ref{pro1}, this implies $\mathbb{P}(\tilde{W}^i > 0) > 0$ and $\mathbb{E}_{\xi}\tilde{W}^i = 1$. From the decomposition of $W^i$ in Proposition \ref{pro4}, we obtain
\[
\mathbb{P}(W^i > 0) > 0 \quad \text{and} \quad \mathbb{E}_{\xi}W^i \ge 1.
\]
Moreover, taking $\mathbb{P}_{\xi,Y}$-expectation in \eqref{2.8} yields
\begin{equation}\label{ew2}
\mathbb{E}_{\xi,Y} W^i = 1 + \sum_{j=0}^{\infty} \frac{\langle Y_j, U_{j+1,\infty} \rangle}{\lambda_{0,j} U_{0,\infty}(i)}.
\end{equation}
Since $\mathbb{E}\log^+ (|Y_0|/\lambda_0) < \infty$, it follows from \cite[Lemma 3.1]{wang} that the series in \eqref{ew2} converges almost surely, whence $\mathbb{E}_{\xi,Y} W^i < \infty$. Consequently, by \eqref{ew1},
\[
\lim_{n \to \infty} \mathbb{E}_{\xi,Y} W_n^i = \mathbb{E}_{\xi,Y} W^i.
\]
An application of the extended Scheff\'{e} theorem (see, e.g., \cite[Lemma 5.1]{wang}) now gives $\mathbb{E}_{\xi,Y} |W^i - W_n^i| \to 0$ $\mathbb{P}$-almost surely.

Conversely, assume $\mathbb{E}_{\xi,Y} |W^i - W_n^i| \to 0$ almost surely. Then, from \eqref{ew1} and \eqref{ew2},
\[
\mathbb{E}_{\xi,Y} W^i = \lim_{n \to \infty} \mathbb{E}_{\xi,Y} W_n^i \ge \lim_{n \to \infty} \mathbb{E}_{\xi} \tilde{W}_n^i = 1 \quad \mathbb{P}\text{-a.s.}
\]
This inequality, together with the decomposition \eqref{2.8}, forces $\mathbb{P}(\tilde{W}^i > 0) > 0$. By Proposition \ref{pro1}, Condition \textbf{H4} must therefore be satisfied. This completes the proof. \qed

\noindent\textbf{Proof of Corollary \ref{cor1}.}

Following an argument analogous to that for \eqref{3.11}, we obtain
\begin{equation}\label{3.15}
\mbb{E}_{\xi}(W^i_{n+1} \mid \mathcal{F}_n) = W_n^i + \frac{\langle \mbb{E}_{\xi} Y_n, U_{n+1,\infty} \rangle}{\lambda_{0,n} U_{0,\infty}(i)}, \qquad n \ge 0.
\end{equation}
Since the environment sequence \((\xi_n)\) is i.i.d., the process \((W_n^i)\) forms a submartingale whenever \(\mbb{E}_{\xi} Y_0 < \infty\) for every realization of \((\xi_n)\).

Iterating \eqref{3.15} yields
\begin{equation}\label{3.13}
\mbb{E}_{\xi} W_n^i = 1 + \sum_{k=0}^{n-1} \frac{\langle \mbb{E}_{\xi} Y_k, U_{k+1,\infty} \rangle}{\lambda_{0,k} U_{0,\infty}(i)}.
\end{equation}
Adopting the same method applied to equation \eqref{ew1}, we complete the proof of $(\mathrm{a})$. 

The $L^1$ convergence under $\mathbb{P}_{\xi}$ follows by applying the argument of Theorem \ref{th1}(c) to the submartingale $(W_n^i)$ under the measure $\mathbb{P}_{\xi}$, with $\mathbb{E}_{\xi}Y_k$ in place of $Y_k$. \qed

\noindent\textbf{Proof of Corollary \ref{cor2}.}
 Taking expectations on both sides of \eqref{3.13} gives
\begin{equation}\label{3.12}
\mbb{E} W_n^i = 1 + \sum_{k=0}^{n-1} \mbb{E}\!\left( \frac{\langle Y_k, U_{k+1,\infty} \rangle}{\lambda_{0,k} U_{0,\infty}(i)} \right).
\end{equation}
Observe that \(\langle Y_k, U_{k+1,\infty} \rangle\) and \(\lambda_k\) depend only on \((\xi_k, \xi_{k+1}, \dots)\). Hence,
\begin{align*}
\mbb{E}_{T^k\xi}\!\left( \frac{\langle Y_k, U_{k+1,\infty} \rangle}{\lambda_{0,k} U_{0,\infty}(i)} \right)
&= \frac{\langle Y_k, U_{k+1,\infty} \rangle}{\lambda_k} \,
\mbb{E}_{T^k\xi}\!\bigl( \lambda_{0,k-1} U_{0,\infty}(i) \bigr)^{-1} \\
&\le \lambda_k^{-1} \langle Y_k, U_{k+1,\infty} \rangle \, \kappa(-1)^k,
\end{align*}
where the inequality follows from \eqref{2.2}. Substituting this into \eqref{3.12} leads to
\begin{align*}
\mbb{E} W_n^i
&\le 1 + \sum_{k=0}^{n-1} \kappa(-1)^k \,
\mbb{E}\!\left( \lambda_k^{-1} \langle Y_k, U_{k+1,\infty} \rangle \right) \\
&= 1 + \sum_{k=0}^{n-1} \kappa(-1)^k \,
\mbb{E}\!\left( \lambda_0^{-1} \langle Y_0, U_{1,\infty} \rangle \right).
\end{align*}
Since \(U_{1,\infty}\) is strictly positive for every realization of \(\xi\), the finiteness of \(\mbb{E}\lambda_0^{-1}\langle Y_0, U_{1,\infty} \rangle\) is equivalent to \(\mbb{E}\lambda_0^{-1}|Y_0| < \infty\). Therefore, if \(\mbb{E}\lambda_0^{-1}|Y_0| < \infty\) and \(\kappa(-1) < 1\), we have \(\sup_{n \ge 0} \mbb{E} W_n^i < \infty\).

For the converse direction, we note that \eqref{2.2} implies
\[
\mbb{E}_{T^k\xi}\!\bigl( \lambda_{0,k-1} U_{0,\infty}(i) \bigr)^{-1}
\ge \mbb{E} \lVert M_{0,k-1} \rVert^{-1}.
\]
Inserting this estimate into \eqref{3.12} gives
\[
\mbb{E} W_n^i \ge 1 + \sum_{k=0}^{n-1}
\bigl( \mbb{E} \lVert M_{0,k-1} \rVert^{-1} \bigr) \,
\mbb{E}\!\left( \lambda_0^{-1} \langle Y_0, U_{1,\infty} \rangle \right),
\]
with the convention \(\vt M_{0,-1}\vt=1\). By the definition \(\kappa(s) = \lim_{n\to\infty} (\mbb{E} \lVert M_{0,n-1} \rVert^{s})^{1/n}\), the condition $\sup_{n \ge 0} \mbb{E} W_n^i < \infty$, together with $\mathbb{P}(|Y_0|=0)<1$, forces \(\mbb{E}\lambda_0^{-1}|Y_0| < \infty\) and \(\kappa(-1) \le 1\). 

The $L^1$ convergence under $\mbb{P}$ also follows from the same argument used for the $L^1$ convergence of $(\tilde{W}_n)$ in Theorem \ref{th1}(c). \qed

\subsection{Proof of Theorem \ref{th4}}\label{sec5.2}

We first prove the sufficiency. Assume that
\[
\kappa(-p) < 1, \qquad \max_{1 \le i, j \le d} \mathbb{E} \left( \frac{Z_1^i(j)}{M_0(i,j)} \right)^p < \infty, \qquad \text{and} \qquad \mathbb{E} \frac{|Y_0|^p}{\lambda_0^p} < \infty.
\]
Recall that $\kappa(-p) = \lim_{n \to \infty} (\mathbb{E} \| M_{0,n-1} \|^{-p})^{1/n}$. Hence, $\mathbb{E} \| M_{0,n-1} \|^{-p} < 1$ for all sufficiently large $n$. By H\"{o}lder's inequality,
\[
\bigl( \mathbb{E} \| M_{0,n-1} \|^{1-p} \bigr)^{1/(p-1)} \le \bigl( \mathbb{E} \| M_{0,n-1} \|^{-p} \bigr)^{1/p},
\]
which implies $\kappa(1-p) < 1$. Consequently, applying Proposition \ref{pro2} yields the $L^p$ convergence $\tilde{W}_n^i \to \tilde{W}^i$ for every $1 \le i \le d$. Doob's $L^p$ maximal inequality (see, e.g., \cite[pp. 252]{durrett}) then gives $\mathbb{E} (\tilde{W}^{i\ast})^p < \infty$.

From the decomposition \eqref{decompose} of $W_n$, we have for all $n \ge 0$,
\[
W_n \le \tilde{W}^{\ast}[(i,\emptyset)] \;+\; \sum_{j=1}^{n} \sum_{k=1}^{d} \sum_{l=1}^{Y_{j-1}(k)} \tilde{W}^{\ast}[(k,0_{j-1}l)] \frac{U_{j,\infty}(k)}{\lambda_{0,j-1} U_{0,\infty}(i)}.
\]
Taking the $L^p$ norm under $\mathbb{P}$ and using its subadditivity, we obtain
\begin{align*}
\| W_n \|_p \le \| \tilde{W}^{\ast}[(i,\emptyset)] \|_p 
+ \sum_{j=1}^{n} \sum_{k=1}^{d} \Biggl[ \mathbb{E} \Biggl( & \bigl( \lambda_{0,j-2} U_{0,\infty}(i) \bigr)^{-p} \\
& \times \Bigl( \lambda_{j-1}^{-1} \sum_{l=1}^{Y_{j-1}(k)} \tilde{W}^{\ast}[(k,0_{j-1}l)] \Bigr)^p \Biggr) \Biggr]^{1/p}.
\end{align*}
Observe that the term $\lambda_{j-1}^{-1} \sum_{l=1}^{Y_{j-1}(k)} \tilde{W}^{\ast}[(k,0_{j-1}l)]$ depends only on $(\xi_{j-1}, \xi_j, \dots)$. Therefore,
\begin{align*}
& \mathbb{E} \Biggl[ \bigl( \lambda_{0,j-2} U_{0,\infty}(i) \bigr)^{-p} \Bigl( \lambda_{j-1}^{-1} \sum_{l=1}^{Y_{j-1}(k)} \tilde{W}^{\ast}[(k,0_{j-1}l)] \Bigr)^p \Biggr] \\
&= \mathbb{E} \Biggl[ \mathbb{E}_{T^{j-1}\xi} \bigl( \lambda_{0,j-2} U_{0,\infty}(i) \bigr)^{-p} \; \mathbb{E}_{\xi} \Bigl( \lambda_{j-1}^{-1} \sum_{l=1}^{Y_{j-1}(k)} \tilde{W}^{\ast}[(k,0_{j-1}l)] \Bigr)^p \Biggr] \\
&\le C \, (\kappa(-p))^{\,j-1} \; \mathbb{E} \Biggl[ \mathbb{E}_{\xi} \Bigl( \lambda_{j-1}^{-1} \sum_{l=1}^{Y_{j-1}(k)} \tilde{W}^{\ast}[(k,0_{j-1}l)] \Bigr)^p \Biggr],
\end{align*}
where the last inequality follows from \eqref{2.2}. Since the variables $\tilde{W}^{\ast}[(k,0_{j-1}l)]$ ($l \ge 1$) are identically distributed under $\mathbb{P}_\xi$, we apply the elementary inequality $(\sum_{i=1}^m x_i)^p \le m^{\,p-1} \sum_{i=1}^m x_i^p$ to get
\begin{align*}
\mathbb{E}_{\xi} \Bigl( \lambda_{j-1}^{-1} \sum_{l=1}^{Y_{j-1}(k)} \tilde{W}^{\ast}[(k,0_{j-1}l)] \Bigr)^p
&\le \lambda_{j-1}^{-p} \, \mathbb{E}_{\xi} \Bigl[ (Y_{j-1}(k))^{\,p-1} \sum_{l=1}^{Y_{j-1}(k)} \bigl( \tilde{W}^{\ast}[(k,0_{j-1}l)] \bigr)^p \Bigr] \\
&\le \lambda_{j-1}^{-p} \, \mathbb{E}_{\xi} \bigl( Y_{j-1}(k)^p \bigr) \; \mathbb{E}_{\xi} (\tilde{W}^{\ast}[(k,0_{j-1}1)])^p.
\end{align*}
Combining the estimates, we arrive at
\[
\sup_n \| W_n \|_p \le \| \tilde{W}^{\ast}[(i,\emptyset)] \|_p \;+\; \sum_{j=1}^{\infty} \sum_{k=1}^{d} C \, (\kappa(-p))^{(j-1)/p} \Bigl[ \mathbb{E} \frac{Y_{j-1}(k)^p}{\lambda_{j-1}^{\,p}} \; \mathbb{E} (\tilde{W}^{\ast}[(k,0_{j-1}1)])^p \Bigr]^{1/p}.
\]
Under the assumed conditions $\kappa(-p)<1$, $\mathbb{E}|Y_0|^p/\lambda_0^p < \infty$, and $\mathbb{E}(\tilde{W}^{i\ast})^p < \infty$, the right-hand side is finite. Hence $\sup_n \| W_n \|_p < \infty$, and the $L^p$ convergence theorem (see, e.g., \cite[pp. 251]{durrett}) yields
\[
\lim_{n \to \infty} \| W_n^i - W^i \|_p = 0.
\]

\noindent\textit{Necessity.} Suppose now that $\sup_n \| W_n \|_p < \infty$. From \eqref{decompose} and the inequality $(\sum_{i=1}^m x_i)^p \ge \sum_{i=1}^m x_i^p$ for non-negative $x_i$, we have
\[
W_n^p \ge \bigl( \tilde{W}_n[(i,\emptyset)] \bigr)^p \;+\; \sum_{j=1}^{n} \sum_{k=1}^{d} \Bigl( \frac{U_{j,\infty}(k)}{\lambda_{0,j-1} U_{0,\infty}(i)} \Bigr)^p \Bigl( \sum_{l=1}^{Y_{j-1}(k)} \tilde{W}_{n-j}[(k,0_{j-1}l)] \Bigr)^p.
\]
Let $\mathbb{E}_{\xi, j-1}(\cdot) = \mathbb{E}_\xi(\cdot \mid Y_{j-1})$ denote the conditional expectation under $\mathbb{P}_\xi$ given $Y_{j-1}$. For fixed $j \ge 1$, note that $\tilde{W}_{n-j}[(k,0_{j-1}l)]$ and $Y_{j-1}(k)$ are independent under $\mathbb{P}_\xi$, and $\mathbb{E}_\xi \tilde{W}_{n-j}[(k,0_{j-1}l)] = 1$. Applying Jensen's inequality,
\begin{align}
\mathbb{E}_{\xi, j-1} \Bigl( \sum_{l=1}^{Y_{j-1}(k)} \tilde{W}_{n-j}[(k,0_{j-1}l)] \Bigr)^p
&\ge \Bigl( \mathbb{E}_{\xi, j-1} \sum_{l=1}^{Y_{j-1}(k)} \tilde{W}_{n-j}[(k,0_{j-1}l)] \Bigr)^p \nonumber\\
&= Y_{j-1}(k)^p \bigl( \mathbb{E}_\xi \tilde{W}_{n-j}[(k,0_{j-1}l)] \bigr)^p \nonumber\\
&= Y_{j-1}(k)^p.\label{4.1}
\end{align}
Consequently, $\mathbb{E}_\xi \bigl( \sum_{l=1}^{Y_{j-1}(k)} \tilde{W}_{n-j}[(k,0_{j-1}l)] \bigr)^p \ge \mathbb{E}_\xi Y_{j-1}(k)^p$. Taking expectations in the lower bound for $W_n^p$ gives
\begin{align*}
\mathbb{E} W_n^p &\ge \mathbb{E} \bigl( \tilde{W}_n[(i,\emptyset)] \bigr)^p \\
&\quad + \sum_{j=1}^{n} \sum_{k=1}^{d} \mathbb{E} \Biggl[ \mathbb{E}_{T^{j-1}\xi} \Bigl( \frac{U_{j,\infty}(k)}{\lambda_{0,j-2} U_{0,\infty}(i)} \Bigr)^p \; \mathbb{E}_{T^{j-1}\xi} \frac{Y_{j-1}(k)^p}{\lambda_{j-1}^{\,p}} \Biggr] \\
&\ge \mathbb{E} \bigl( \tilde{W}_n[(i,\emptyset)] \bigr)^p \;+\; C \sum_{j=1}^{n} \sum_{k=1}^{d} \mathbb{E} \| M_{0,j-2} \|^{-p} \; \mathbb{E} \frac{Y_0(k)^p}{\lambda_0^{\,p}},
\end{align*}
where the last line uses \eqref{2.2} (with the convention $\| M_{0,-1} \| = 1$). Because $\sup_n \mathbb{E} W_n^p < \infty$ and $\mathbb{P}(|Y_0|=0) < 1$, we must have $\mathbb{E} (\tilde{W}_n[(i,\emptyset)])^p < \infty$ and
\[
\kappa(-p) \le 1, \qquad \mathbb{E} \frac{|Y_0|^p}{\lambda_0^p} < \infty.
\]
Finally, Proposition \ref{pro2} shows that $\mathbb{E} (\tilde{W}_n[(i,\emptyset)])^p < \infty$ implies
\[
\max_{1 \le i, j \le d} \mathbb{E} \Bigl( \frac{Z_1^i(j)}{M_0(i,j)} \Bigr)^p < \infty \quad \text{and} \quad \kappa(1-p) < 1.
\]
This completes the proof of the necessity, and therefore of the theorem.

\subsection{Proof of Theorem \ref{th5}}
The proof follows an argument analogous to that of Theorem~\ref{th4} and is divided into two parts.

\noindent\textit{Sufficiency.} Assume that $\kappa(-p) < 1$ and $\mathbb{E}|Y_{0}|^p \lambda_0^{-p} < \infty$. Applying the subadditivity of the function $x \mapsto x^p$ (for $0<p<1$) to the decomposition \eqref{decompose} of $W_n^i$ yields
\begin{align}
\| W^i_n \|_p \le \| \tilde{W}_n[(i,\emptyset)] \|_p 
+ \sum_{j=1}^{n} \sum_{k=1}^{d} \Biggl[ \mathbb{E} \Biggl( & \bigl( \lambda_{0,j-2} U_{0,\infty}(i) \bigr)^{-p} \nonumber\\
& \times \Bigl( \lambda_{j-1}^{-1} \sum_{l=1}^{Y_{j-1}(k)} \tilde{W}_{n-j}[(k,0_{j-1}l)] \Bigr)^p \Biggr) \Biggr]^{1/p}.\label{4.2}
\end{align}
Following a similar estimation procedure as in \eqref{4.1}, and using the inequality $\mathbb{E} X^p \le (\mathbb{E} X)^p$ for $0<p<1$ and a nonnegative random variable $X$, we obtain
\[
\mathbb{E}_{\xi, j-1} \Bigl( \sum_{l=1}^{Y_{j-1}(k)} \tilde{W}_{n-j}[(k,0_{j-1}l)] \Bigr)^p
\le Y_{j-1}(k)^p.
\]
Consequently,
\begin{align*}
& \mathbb{E} \Biggl[ \bigl( \lambda_{0,j-2} U_{0,\infty}(i) \bigr)^{-p} \Bigl( \lambda_{j-1}^{-1} \sum_{l=1}^{Y_{j-1}(k)} \tilde{W}_{n-j}[(k,0_{j-1}l)] \Bigr)^p \Biggr] \\
&= \mathbb{E} \Biggl[ \mathbb{E}_{T^{j-1}\xi} \bigl( \lambda_{0,j-2} U_{0,\infty}(i) \bigr)^{-p} \; \mathbb{E}_{\xi} \Bigl( \lambda_{j-1}^{-1} \sum_{l=1}^{Y_{j-1}(k)} \tilde{W}_{n-j}[(k,0_{j-1}l)] \Bigr)^p \Biggr] \\
&\le C \, \left(\kappa(-p)\right)^{\,j-1} \; \mathbb{E}\left(\frac{Y_{0}(k)}{\lambda_0}\right)^p.
\end{align*}
Note that $\mathbb{E}(\tilde{W}_n[(i,\emptyset)])^p \le (\mathbb{E}\tilde{W}_n[(i,\emptyset)])^p = 1$. Substituting this estimate into \eqref{4.2} gives
\[
\| W^i_n \|_p \le 1 + \sum_{j=1}^{n} \sum_{k=1}^{d} C \, \left(\kappa(-p)\right)^{\,j-1} \; \mathbb{E}\left(\frac{Y_{0}(k)}{\lambda_0}\right)^p.
\]
Under the assumptions $\kappa(-p)<1$ and $\mathbb{E}|Y_{0}|^p\lambda_0^{-p}<\infty$, the right-hand side is uniformly bounded in $n$, establishing $\sup_n \|W_n^i\|_p < \infty$.

\noindent\textit{Necessity.} Conversely, suppose $\sup_n \| W_n \|_p < \infty$. From \eqref{decompose}, we have the lower bound
\[
W_n^i \ge \sum_{k=1}^{d} \sum_{l=1}^{Y_{n-1}(k)} \tilde{W}_{0}[(k,0_{n-1}l)] \frac{U_{n,\infty}(k)}{\lambda_{0,n-1} U_{0,\infty}(i)}
= (\lambda_{0,n-2}U_{0,\infty}(i))^{-1} \frac{\langle Y_{n-1}, U_{n,\infty}\rangle}{\lambda_{n-1}}.
\]
Taking expectations and using the independence structure yields
\[
\mathbb{E} (W_n^i)^p \ge \mathbb{E} \Bigl[ \mathbb{E}_{T^{n-1}\xi} (\lambda_{0,n-2}U_{0,\infty}(i))^{-p} \; \mathbb{E}_{\xi} \Bigl( \frac{\langle Y_{0}, U_{1,\infty}\rangle}{\lambda_{0}} \Bigr)^p \Bigr].
\]
Since $U_{1,\infty} > 0$ almost surely and $\kappa(-p) = \lim_{n \to \infty} (\mathbb{E}\|M_{0,n-1}\|^{-p})^{1/n}$, the uniform boundedness of $\mathbb{E} (W_n^i)^p$ forces
\[
\kappa(-p) \le 1 \quad \text{and} \quad \mathbb{E}|Y_{0}|^p\lambda_0^{-p} < \infty.
\]
This completes the proof of the necessity, and thus of the theorem.

\section*{Acknowledgement}
\par

\end{document}